\documentclass{amsart}

\usepackage{amssymb}
\usepackage{hyperref}
\usepackage{txfonts}
\usepackage{amsmath,amscd}
\usepackage[enableskew]{youngtab}
\usepackage[smalltableaux]{ytableau}
\usepackage[all,cmtip]{xy}

\newtheorem{theorem}{Theorem}[section]
\newtheorem{lemma}[theorem]{Lemma}
\newtheorem{proposition}[theorem]{Proposition}
\newtheorem{corollary}[theorem]{Corollary}

\theoremstyle{definition}

\theoremstyle{remark}
\newtheorem{remark}[theorem]{Remark}

\numberwithin{equation}{section}

\def\NN{\mathbb{N}} \def\ZZ{{\mathbb Z}} \def\CC{{\mathbb C}} \def\FF{{\mathbb F}}
\def\NH{\mathbf{H}} \def\tNH{\widetilde{\mathbf{H}}}
\def\cB{\mathcal{B}} \def\SS{{\mathfrak S}} \def\MM{{\ds M}} 
\def\px{X} \def\bx{{\mathbf x}} \def\bh{\mathbf{h}} \def\bs{{\mathbf s}} 
\def\QSym{\mathrm{QSym}} \def\NSym{\mathbf{NSym}} \def\Sym{\mathrm{Sym}}
\def\C{\mathbf{C}} \def\P{\mathbf{P}}
\def\des{{\rm des}} \def\maj{{\rm maj}} \def\inv{{\rm inv}} \def\Inv{{\rm Inv}}
\def\cleq{{\preccurlyeq}} 
\def\ps{\mathbf{ps}} \def\pp{{\mathbf p}} \def\pib{{\overline\pi}}
\def\Com{\mathrm{Com}} \def\Par{\mathrm{Par}}
\def\ds{\rule{0pt}{1.5ex}}

\begin{document}

\title{$0$-Hecke algebra action on the Stanley-Reisner ring of the Boolean algebra}
\author{Jia Huang}
\address{School of Mathematics, University of Minnesota, Minneapolis, MN 55455, USA}
%\curraddr{}
\email{huang338@umn.edu}
\thanks{The author is grateful to Victor Reiner for providing valuable suggestions and support from NSF grant DMS-1001933. He also thanks Ben Braun and Jean-Yves Thibon for helpful conversations and email correspondence.}
\maketitle

\begin{abstract}
We define an action of the $0$-Hecke algebra of type A on the Stanley-Reisner ring of the Boolean algebra. By studying this action we obtain a family of multivariate noncommutative symmetric functions, which specialize to the noncommutative Hall-Littlewood symmetric functions and their $(q,t)$-analogues introduced by Bergeron and Zabrocki, and to a more general family of noncommutative symmetric functions having parameters associated with paths in binary trees introduced recently by Lascoux, Novelli, and Thibon. We also obtain multivariate quasisymmetric function identities, which specialize to results of Garsia and Gessel on generating functions of multivariate distributions of permutation statistics.
%$\maj(w)$, $\des(w)$, $\maj(w^{-1})$, and $\des(w^{-1})$ for all permutations $w$.
\end{abstract}

\section{Introduction}

Let $\FF$ be an arbitrary field. The symmetric group $\SS_n$ naturally acts on the polynomial ring $\FF[\px]:=\FF[x_1,\ldots,x_n]$ by permuting the variables $x_1,\ldots,x_n$. The invariant algebra $\FF[\px]^{\SS_n}$, which consists of all the polynomials fixed by this $\SS_n$-action, is a polynomial algebra generated by the elementary symmetric functions $e_1,\ldots,e_n$. The coinvariant algebra $\FF[\px]/(\FF[\px]^{\SS_n}_+)$, with $(\FF[\px]^{\SS_n}_+)=(e_1,\ldots,e_n)$, is a vector space of dimension $n!$ over $\FF$, and if $n$ is not divisible by the characteristic of $\FF$ then it carries the regular representation of $\SS_n$. A well known basis for $\FF[\px]/(\FF[\px]^{\SS_n}_+)$ consists of the descent monomials. Garsia~\cite{Garsia} obtained this basis by transferring a natural basis from the Stanley-Reisner ring $\FF[\cB_n]$ of the Boolean algebra $\cB_n$ to the polynomial ring $\FF[\px]$. Here the Boolean algebra $\cB_n$ is the set of all subsets of $[n]:=\{1,\ldots,n\}$ partially ordered by inclusion, and $\FF[\cB_n]$ is the quotient of the polynomial algebra $\FF \left[ y_{\ds A}: A\subseteq[n] \right]$ by the ideal $\left( y_{\ds A} y_{\ds B}: \textrm{ $A$ and $B$ are incomparable in }\cB_n \right)$. 

The $0$-Hecke algebra $H_n(0)$ (of type $A$) is a deformation of the group algebra of $\SS_n$. It acts on $\FF[\px]$ by the Demazure operators, also known as the isobaric divided difference operators, having the same invariant algebra as the $\SS_n$-action on $\FF[\px]$. In our earlier work~\cite{H}, we showed that the coinvariant algebra $\FF[\px]/(\FF[\px]^{\SS_n}_+)$ is also isomorphic to the regular representation of $H_n(0)$, for any field $\FF$, by constructing another basis for $\FF[\px]/(\FF[\px]^{\SS_n}_+)$ which consists of certain polynomials whose leading terms are the descent monomials. This and the previously mentioned connection between the Stanley-Reisner ring $\FF[\cB_n]$ and the polynomial ring $\FF[\px]$ motivate us to define an $H_n(0)$-action on $\FF[\cB_n]$.

It turns out that our $H_n(0)$-action on $\FF[\cB_n]$ has similar definition and properties to the $H_n(0)$-action on $\FF[X]$. It preserves the $\NN^{n+1}$-multigrading of $\FF[\cB_n]$ and has invariant algebra equal to a polynomial algebra $\FF[\Theta]$, where $\Theta$ is the set of \emph{rank polynomials} $\theta_i$ (the usual analogue of $e_i$ in $\FF[\cB_n]$). We show that the $H_n(0)$-action is $\Theta$-linear and thus reduces to the coinvariant algebra $\FF[\cB_n]/(\Theta)$. We will see that $\FF[\cB_n]/(\Theta)$ carries the regular representation of $H_n(0)$. 

Furthermore, using the $H_n(0)$-action on $\FF[\cB_n]$ we obtain a noncommutative analogue (see Theorem~\ref{thm1}) for the remarkable representation theoretic interpretation of the Hall-Littlewood symmetric functions by the $\SS_n$-action on the cohomology ring of the \emph{Springer fibers} (see e.g. Hotta-Springer~\cite{HottaSpringer} and Garsia-Procesi~\cite{GarsiaProcesi}). In the analogous case of $H_n(0)$ acting on $\FF[\px]$, one only has a partial noncommutative analogue for hooks in \cite{H}. 

%There is a classic correspondence via the Frobenius characteristic map between the Grothendieck group of finite dimensional (complex) representations of $\SS_n$ and the self-dual Hopf algebra of symmetric functions. 
To state this result, we first recall that every finite dimensional (complex) $\SS_n$-representation is a direct sum of simple (i.e. irreducible) $\SS_n$-modules, and the simple $\SS_n$-modules are indexed by partitions $\lambda$ of $n$, which correspond to the Schur functions $s_\lambda$ via the \emph{Frobenius characteristic map}. Hotta-Springer~\cite{HottaSpringer} and Garsia-Procesi~\cite{GarsiaProcesi} discovered that the cohomology ring of the Springer fibers are graded $\SS_n$-modules corresponding to the modified Hall-Littlewood symmetric functions via the Frobenius characteristic map.

Krob and Thibon~\cite{KrobThibon} introduced two characteristic maps for $H_n(0)$-representations, which we call the \emph{quasisymmetric characteristic} and the \emph{noncommutative characteristic}. The simple $H_n(0)$-modules are indexed by compositions $\alpha$ of $n$ and correspond to the \emph{fundamental quasisymmetric functions} $F_\alpha$ via the quasisymmetric characteristic; the projective indecomposable $H_n(0)$-modules are also indexed by compositions $\alpha$ of $n$ and correspond to the \emph{noncommutative ribbon Schur functions} $\bs_\alpha$ via the noncommutative characteristic. %One easily gets multigraded versions for multigraded $H_n(0)$-modules.

Using an analogue of the nabla operator, Bergeron and Zabrocki~\cite{BZ} introduced a noncommutative modified Hall-Littlewood symmetric function $\tNH_\alpha(\bx;t)$ and a $(q,t)$-analogue $\tNH_\alpha(\bx;q,t)$ for all compositions $\alpha$. Our theorem below provides a representation theoretic interpretation for them.

\begin{theorem}\label{thm1}
Let $\alpha$ be a composition of $n$. Then there exists a homogeneous $H_n(0)$-invariant ideal $I_\alpha$ of the multigraded algebra $\FF[\cB_n]$ such that the quotient algebra $\FF[\cB_n]/I_\alpha$ becomes a projective $H_n(0)$-module with multigraded noncommutative characteristic equal to
\[
\tNH_\alpha(\bx;t_1,\ldots,t_{n-1}):= \sum_{\beta\cleq\alpha} \underline t^{D(\beta)} \bs_\beta \quad {\rm inside} \quad \NSym[t_1,\ldots,t_{n-1}].
\]
Moreover, one has $\tNH_\alpha(\bx;t,t^2,\ldots,t^{n-1}) = \tNH_\alpha(\bx;t)$, and we obtain $\tNH_\alpha(\bx;q,t)$ from $\tNH_{1^n}(\mathbf{x};t_1,\ldots,t_{n-1})$ by taking $t_i=t^i$ for all $i\in D(\alpha)$, and $t_i=q^{n-i}$ for all $i\in [n-1]\setminus D(\alpha)$.
\end{theorem}

Here $D(\alpha)$ is the set of the partial sums of the composition $\alpha$, the notation $\beta\cleq\alpha$ means $\alpha$ and $\beta$ are compositions of $n$ with $D(\beta)\subseteq D(\alpha)$, and $\underline t^S$ denotes the product $\prod_{i\in S} t_i$ over all elements $i$ in a multiset $S$, including the repeated ones. We also generalize the basic properties of $\tNH_\alpha(\bx;t)$ given in \cite{BZ} to the multivariate $\tNH_\alpha(\bx;t_1,\ldots,t_{n-1})$.

Note that $\tNH_{1^n}(\mathbf{x};t_1,\ldots,t_{n-1})$ is the multigraded noncommutative characteristic of the coinvariant algebra $\FF[\cB_n]/(\Theta)$, from which one sees that $\FF[\cB_n]/(\Theta)$ carries the regular representation of $H_n(0)$. Specializations of $\tNH_{1^n}(\mathbf{x};t_1,\ldots,t_{n-1})$ include not only $\tNH_\alpha(\bx;q,t)$, but also a more general family of noncommutative symmetric functions depending on parameters associated with paths in binary trees introduced recently by Lascoux, Novelli, and Thibon~\cite{LascouxNovelliThibon}.

Next we study the quasisymmetric characteristic of $\FF[\cB_n]$. We combine the usual $\NN^{n+1}$-multigrading of $\FF[\cB_n]$ with the length filtration of $H_n(0)$ and obtain an $\NN\times \NN^{n+1}$-multigraded quasisymmetric characteristic for $\FF[\cB_n]$.

\begin{theorem}\label{thm2}
The $\NN\times\NN^{n+1}$-multigraded quasisymmetric characteristic of $\FF[\cB_n]$ is
%\begin{eqnarray*}
\[
\sum_{k\geq0} \sum_{\alpha\in\Com(n,k+1)} \underline t^{D(\alpha)} \sum_{w\in\SS^\alpha}  q^{\inv(w)} F_{D(w^{-1})} 
=\sum_{w\in\SS_n} \frac{ q^{\inv(w)} \underline t^{D(w)} F_{D(w^{-1})} }
{\prod_{0\leq i\leq n} (1-t_i)}  
= \sum_{k\geq0} \sum_{\pp\in[k+1]^n} t_{p'_1} \cdots t_{p'_k} q^{\inv(\pp)} F_{D(\pp)}.
\]
%\end{eqnarray*}
\end{theorem}

Here we identify $F_I$ with $F_\alpha$ if $D(\alpha)=I\subseteq[n-1]$. The set $\Com(n,k)$ consists of all \emph{weak compositions of $n$ with length $k$}, i.e. all sequences $\alpha=(\alpha_1,\ldots,\alpha_k)$ of $k$ nonnegative integers such that $|\alpha|:=\alpha_1+\cdots+\alpha_k=n$. The \emph{descent multiset} of the weak composition $\alpha$ is defined to be the \emph{multiset} $D(\alpha):=\{\alpha_1,\alpha_1+\alpha_2,\alpha_1+\cdots+\alpha_{k-1}\}$. We also define $\SS^\alpha:=\left\{w\in\SS_n:D(w)\subseteq D(\alpha)\right\}$. %where $\underline{D}(\alpha)$ is the underlying set of the multiset $D(\alpha)$.
The set $[k+1]^n$ consists of all words of length $n$ on the alphabet $[k+1]$. Given $\pp=(p_1,\ldots,p_n)\in[k+1]^n$, we write $p'_i:=\#\{j:p_j\leq i\}$,
$\inv(\pp):=\#\{(i,j):1\leq i<j\leq n:p_i>p_j\},$
and $D(\pp):=\{i:p_i>p_{i+1}\}$.

Let $\ps_{q;\ell}(F_\alpha):=F_\alpha(1,q,q^2,\ldots,q^{\ell-1},0,0,\ldots)$. Then applying $\sum_{\ell\geq0} u_1^{\ell}\ps_{q_1;\ell+1}$ and the specialization $t_i=q_2^iu_2$ for all $i=0,1,\ldots,n$ to Theorem~\ref{thm2}, we recover the following result of Garsia and Gessel~\cite[Theorem~2.2]{GarsiaGessel} on the generating function of multivariate distribution of five permutation statistics:
\[
\frac{ \sum_{w\in\SS_n} q_0^{\inv(w)} q_1^{\maj(w^{-1})} u_1^{\des(w^{-1})} q_2^{\maj(w)} u_2^{\des(w)} } {(u_1;q_1)_n (u_2;q_2)_n }
= \sum_{\ell,k\geq0} u_1^\ell u_2^k \sum_{ (\lambda,\mu)\in B(\ell,k) } q_0^{\inv(\mu)} q_1^{|\lambda|} q_2^{|\mu|}.
\]
Here $(u;q)_n:=\prod_{0\leq i\leq n}(1-q^iu)$, the set $B(\ell,k)$ consists of all pairs of weak compositions $\lambda=(\lambda_1,\ldots,\lambda_n)$ and $\mu=(\mu_1,\ldots,\mu_n)$ satisfying the conditions $\ell\geq\lambda_1\geq\cdots\geq\lambda_n$, $\max\{\mu_i:1\leq i\leq n\}\leq k$, and $\lambda_i=\lambda_{i+1}\Rightarrow \mu_i\geq\mu_{i+1}$ (such pairs $(\lambda,\mu)$ are sometimes called \emph{bipartite partitions}), and $\inv(\mu):=\#\{(i,j):1\leq i<j\leq n,\ \mu_i>\mu_j\}$. Some further specializations of Theorem~\ref{thm2} imply identities of MacMahon-Carlitz and Adin-Brenti-Roichman (see \S\ref{sec:Ch}). 

Now let $W$ be any finite Coxeter group with Coxeter complex $\Delta(W)$ and Hecke algebra $H_W(q)$. We can generalize our $H_n(0)$-action on $\FF[\cB_n]$ to an $H_W(q)$-action on the Stanley-Reisner ring of $\Delta(W)$. Due to technicality we only give a sketch for the main results on this action in Section~\ref{sec:Remarks}.

The structure of this paper is given below. Section~\ref{sec:H0} reviews the representation theory of the $0$-Hecke algebra.  Section~\ref{sec:SRBoolean} studies the Stanley-Reisner ring of the Boolean algebra. Section~\ref{sec:H0SR} defines a $0$-Hecke algebra action on the Stanley-Reisner ring of the Boolean algebra and establishes Theorem~\ref{thm1} and Theorem~\ref{thm2}. Finally we give some remarks and questions for future research in Section~\ref{sec:Remarks}.

\section{Representation theory of the $0$-Hecke algebra}\label{sec:H0}

We review the representation theory of the $0$-Hecke algebra in this section.

\subsection{Symmetric group}
The symmetric group $\SS_n$ is generated by $s_1,\ldots,s_{n-1}$, where $s_i:=(i,i+1)$ is the simple transposition of $i$ and $i+1$, with relations
\[
\left\{\begin{array}{ll}
s_i^2=1, & 1\leq i\leq n-1,\\
s_is_j=s_js_i, & 1\leq i,j\leq n-1,\ |i-j|>1,\\
s_is_{i+1}s_i = s_{i+1}s_is_{i+1}, & 1\leq i\leq n-2. 
\end{array}\right.
\]
Let $w$ be a permutation in $\SS_n$. If an expression $w=s_{i_1}\cdots s_{i_k}$ is the shortest one among all such expressions, then it is said to be \emph{reduced} and $k$ is the \emph{length} $\ell(w)$ of $w$. The \emph{descent set} of $w$ is $D(w):=\{i\in[n-1]: \ell(ws_i)<\ell(w)\}$ and its elements are called the \emph{descents} of $w$. One has $\ell(w)=\inv(w)=\#\Inv(w)$ where 
\[
\Inv(w):= \{(i,j):1\leq i<j\leq n,\ w(i)>w(j) \}
\]
is the set of inversion pairs of $w$. One also has 
\[
D(w) = \{ i: 1\leq i\leq n-1,\ w(i)>w(i+1) \}.
\]
We denote $\des(w):=|D(w)|$ and $\maj(w):=\sum_{i\in D(w)} i$.

It is often convenient to use compositions when studying $\SS_n$. A \emph{composition} is a sequence $\alpha=(\alpha_1,\ldots,\alpha_\ell)$ of positive integers $\alpha_1,\ldots,\alpha_\ell$. The \emph{length} of $\alpha$ is $\ell(\alpha):=\ell$ and the \emph{size} of $\alpha$ is $|\alpha|:=\alpha_1+\cdots+\alpha_\ell$. If the size of $\alpha$ is $n$ then we say that $\alpha$ is a composition of $n$ and write $\alpha\models n$. 

Let $\alpha=(\alpha_1,\ldots,\alpha_\ell)$ be a composition of $n$. We write $\sigma_j:=\alpha_1+\cdots+\alpha_j$ for $j=0,1,\ldots,\ell$; in particular, $\sigma_0=0$ and $\sigma_\ell=n$. The \emph{descent set} of $\alpha$ is $D(\alpha):=\{\sigma_1,\ldots,\sigma_{\ell-1}\}$. The map $\alpha\mapsto D(\alpha)$ is a bijection between compositions of $n$ and subsets of $[n-1]$. Let $\alpha^c$ be the composition of $n$ with $D(\alpha^c)=[n-1]\setminus D(\alpha)$. %Denote 

The \emph{parabolic subgroup} $\SS_\alpha$ is the subgroup of $\SS_n$ generated by $\{s_i:i\in D(\alpha^c) \}$; a permutation $w\in \SS_n$ belongs to $\SS_\alpha$ if and only if it permutes $\sigma_{i-1}+1,\ldots,\sigma_i$ for all $i=1,\ldots,\ell$. The set of all minimal $\SS_\alpha$-coset representatives is $\SS^\alpha:=\{w\in \SS_n: D(w)\subseteq D(\alpha)\}$.

Given a composition $\alpha$ of $n$, the \emph{descent class} of $\alpha$ consists of the permutations in $\SS_n$ with descent set equal to $D(\alpha)$, and turns out to be an interval under the (left) weak order of $\SS_n$, which is denoted by $[w_0(\alpha),w_1(\alpha)]$. One sees that $w_0(\alpha)$ is the longest element of the parabolic subgroup $\SS_{\alpha^c}$, and $w_1(\alpha)$ is the longest element in $\SS^{\alpha}$ (c.f. Bj\"orner and Wachs~\cite[Theorem~6.2]{BjornerWachs}).

\subsection{$0$-Hecke algebra}

The $0$-Hecke algebra $H_n(0)$ is a deformation of the group algebra of $\SS_n$. It is defined as the $\FF$-algebra generated by $\overline\pi_1,\ldots,\overline\pi_{n-1}$ with relations
\[
\left\{\begin{array}{ll}
\pib_i^2=-\pib_i, & 1\leq i\leq n-1,\\
\pib_i\pib_j=\pib_j\pib_i, & 1\leq i,j\leq n-1,\ |i-j|>1,\\
\pib_i\pib_{i+1}\pib_i = \pib_{i+1}\pib_i\pib_{i+1}, & 1\leq i\leq n-2. 
\end{array}\right.
\]
Let $\pi_i:=\overline\pi_i+1$. Then $\pi_1,\ldots,\pi_{n-1}$ form another generating set for $H_n(0)$, with relations
\[
\left\{\begin{array}{ll}
\pi_i^2=\pi_i, & 1\leq i\leq n-1,\\
\pi_i\pi_j=\pi_j\pi_i, & 1\leq i,j\leq n-1,\ |i-j|>1,\\
\pi_i\pi_{i+1}\pi_i = \pi_{i+1}\pi_i\pi_{i+1}, & 1\leq i\leq n-2. 
\end{array}\right.
\]
If $w=s_{i_1}\cdots s_{i_k}$ is a reduced expression then  $\overline\pi_w:=\overline\pi_{i_1}\cdots\overline\pi_{i_k}$ and $\pi_w:=\pi_{i_1}\cdots\pi_{i_k}$ are well defined. Both $\{\overline\pi_w:w\in\SS_n\}$ and $\{\pi_w:w\in\SS_n\}$ are $\FF$-bases for $H_n(0)$. One can check that $\pi_w$ equals the sum of $\overline\pi_u$ over all $u$ less than or equal to $w$ in the Bruhat order of $\SS_n$. In particular, one has
\[ \pi_{w_0(\alpha)}=\sum_{u\in\SS_{\alpha^c}} \overline\pi_u,\quad \forall \alpha\models n. \]

Norton~\cite{Norton} decomposed the $0$-Hecke algebra $H_n(0)$ into the following direct sum 
\[
H_n(0)=\bigoplus_{\alpha\models n} \P_\alpha.
\]
Each summand $\P_\alpha:=H_n(0)\cdot \overline\pi_{w_0(\alpha)}\pi_{w_0(\alpha^c)}$ has an $\FF$-basis $\left\{ \overline\pi_w \pi_{w_0(\alpha^c)}: w\in [w_0(\alpha),w_1(\alpha)] \right\}$. Its \emph{radical} ${\rm rad}\,\P_\alpha$ is defined as the intersection of all maximal $H_n(0)$-submodules in general, and turns out to be the unique maximal $H_n(0)$-submodule spanned by $\{\pib_w\pi_{w_0(\alpha^c)}:w\in(w_0(\alpha),w_1(\alpha)]\}$ in this case. Although $\P_\alpha$ itself is not necessarily simple, its \emph{top} $\C_\alpha:= \P_\alpha /\,{\rm rad}\,\P_\alpha$ is a one-dimensional  simple $H_n(0)$-module with the action of $H_n(0)$ given by
\[
\pib_i=\left\{\begin{array}{ll}
-1,& {\rm if}\ i\in D(\alpha),\\
0,& {\rm if}\ i\notin D(\alpha).
\end{array}\right.
\]
It follows from general representation theory of algebras (see e.g. \cite[\S I.5]{ASS}) that $\{\P_\alpha: \alpha\models n\}$ is a complete list of pairwise non-isomorphic projective indecomposable $H_n(0)$-modules and $\{\C_\alpha: \alpha \models n\}$ is a complete list of pairwise non-isomorphic simple $H_n(0)$-modules.

\subsection{Quasisymmetric functions and noncommutative symmetric functions}\label{sec:QSymNSym}
The Hopf algebra $\QSym$ of quasisymmetric functions is a free $\ZZ$-module on the \emph{monomial quasisymmetric functions}
\[
M_\alpha:=\sum_{i_1>\cdots>i_\ell\geq1} x_{i_1}^{\alpha_1}\cdots x_{i_\ell}^{\alpha_\ell}
\]
for all compositions $\alpha=(\alpha_1,\ldots,\alpha_\ell)$. Another free $\ZZ$-basis consists of the \emph{fundamental quasisymmetric functions}
\[
F_\alpha:=\sum_{\alpha\cleq\beta} M_\beta = \sum_{\substack{ i_1\geq\cdots\geq i_n\geq1\\ j\in D(\alpha)\Rightarrow i_j>i_{j+1} }} x_{i_1}\cdots x_{i_n}
\]
for all compositions $\alpha$. Here $\alpha\cleq\beta$ means that $\alpha$ and $\beta$ are both compositions of $n$ with $D(\alpha)\subseteq D(\beta)$, or in other words, $\alpha$ is refined by $\beta$. Since $\alpha\mapsto D(\alpha)$ is a bijection, we sometimes write $F_I:=F_\alpha$ if $n$ is clear from the context and $I=D(\alpha)\subseteq[n-1]$. One has an inclusion $\Sym\subseteq \QSym$ of Hopf algebras, where $\Sym$ is the \emph{ring of symmetric functions}.

The reader might notice that the above definition for $M_\alpha$ and $F_\alpha$ is slightly different from the standard one, as the inequalities of the subscripts are reversed. This difference is certainly not essential, and the definition given here has the advantage that the principal specialization $F_\alpha(1,x,x^2,\ldots)$ involves $\maj(\alpha):=\sum_{i\in D(\alpha)} i$ rather than $\maj(\overleftarrow{\alpha})$, where $\overleftarrow\alpha:=(\alpha_\ell,\ldots,\alpha_1)$ if $\alpha=(\alpha_1,\ldots,\alpha_\ell)$. We will use this in \S\ref{sec:Ch}.

The Hopf algebra $\NSym$ is the free associative algebra $\ZZ\langle\bh_1,\bh_2,\ldots\rangle$ where
\[
\bh_k:=\sum_{1\leq i_1\leq\cdots\leq i_k}\bx_{i_1}\cdots \bx_{i_k}. %\quad \forall\,k\geq1.
\]
It has a $\ZZ$-basis of the \emph{complete homogeneous noncommutative symmetric functions} $\bh_\alpha:=\bh_{\alpha_1}\cdots \bh_{\alpha_\ell}$ for all compositions $\alpha=(\alpha_1,\ldots,\alpha_\ell)$. Another free $\ZZ$-basis consists of the \emph{noncommutative ribbon Schur functions}
\[
\bs_\alpha:=\sum_{\beta\cleq\alpha}(-1)^{\ell(\alpha)-\ell(\beta)}\bh_\beta
\]
for all compositions $\alpha$. If $\alpha=(\alpha_1,\ldots,\alpha_\ell)$ and $\beta=(\beta_1,\ldots,\beta_k)$ are two compositions then $\bs_\alpha \bs_\beta = \bs_{\alpha\beta}+\bs_{\alpha\rhd\beta}$ where 
\[
\alpha\beta:=(\alpha_1,\ldots,\alpha_\ell,\beta_1,\ldots,\beta_k),
\]
\[
\alpha\rhd\beta:=(\alpha_1,\ldots,\alpha_{\ell-1},\alpha_\ell+\beta_1,\beta_2,\ldots,\beta_k).
\]
There is a projection $\NSym\twoheadrightarrow\Sym$ of Hopf algebras given by taking the commutative image, i.e. sending the noncommutative variables $\bx_i$ to the commutative variables $x_i$ for all elements in $\NSym$.

The duality between $\QSym$ and $\NSym$ is given by the pairing $\langle M_\alpha, \bh_\beta\rangle = \langle F_\alpha, \bs_\beta\rangle :=\delta_{\alpha,\,\beta}$. 

\subsection{Characteristic maps}\label{sec:char}
Now we recall the classic correspondence between (complex) $\SS_n$-representations and symmetric functions, and a similar correspondence by Krob and Thibon~\cite{KrobThibon} for $H_n(0)$-representations (over any field $\FF$). See also Bergeron and Li~\cite{BergeronLi}.

Let $A$ be an $\FF$-algebra and let $\mathcal C$ be a category of some finitely generated $A$-modules. The \emph{Grothendieck group of $\mathcal C$} is defined as the abelian group $F/R$, where $F$ is the free abelian group on the isomorphism classes $[M]$ of the $A$-modules $M$ in $\mathcal C$, and $R$ is the subgroup of $F$ generated by the elements $[M]-[L]-[N]$ corresponding to all exact sequences $0\to L\to M\to N\to0$ of $A$-modules in $\mathcal C$. Note that if $A$ is semisimple, or if $L,M,N$ are all projective $A$-modules, then $0\to L\to M\to N\to0$ implies $M\cong L\oplus N$. We often identify an $A$-module in $\mathcal C$ with the corresponding element in the Grothendieck group of $\mathcal C$.

Denote by $G_0(\SS_n)$ the Grothendieck group of the category of all finitely generated $\CC\SS_n$-modules. The simple $\CC\SS_n$-modules $S_\lambda$ are indexed by partitions $\lambda$ of $n$, and every finitely generated $\CC\SS_n$-module is a direct sum of simple $\SS_n$-modules. Thus $G_0(\SS_n)$ is a free abelian group on the isomorphism classes $[S_\lambda]$ for all partitions $\lambda$ of $n$. The tower of groups $\SS_\bullet: \SS_0\subset\SS_1\subset\SS_2\subset\cdots$ has a Grothendieck group 	
\[
G_0(\SS_\bullet) := \bigoplus_{n\geq0} G_0(\SS_n)
\]
which turns out to be a self-dual Hopf algebra with product and coproduct given by induction and restriction of representations. The \emph{Frobenius characteristic map} ch is defined by sending a simple $S_\lambda$ to the Schur function $s_\lambda$, giving a Hopf algebra isomorphism between the Grothendieck group $G_0(\SS_\bullet)$ and the Hopf algebra $\Sym$ of symmetric functions. 

The Grothendieck group of the category of all finitely generated $H_n(0)$-modules is denoted by $G_0(H_n(0))$, and the Grothendieck group of the category of finitely generated projective $H_n(0)$-modules is denoted by $K_0(H_n(0))$. By the result of Norton~\cite{Norton}, one has 
\[
G_0(H_n(0))=\bigoplus_{\alpha\models n} \ZZ\cdot [\C_\alpha],\quad
K_0(H_n(0))=\bigoplus_{\alpha\models n} \ZZ\cdot [\P_\alpha].
\]
The tower of algebras $H_\bullet(0): H_0(0)\subset H_1(0)\subset H_2(0)\subset\cdots$ has two Grothendieck groups
\[
G_0(H_\bullet(0)):=\bigoplus_{n\geq0}G_0(H_n(0)),\quad
K_0(H_\bullet(0)):=\bigoplus_{n\geq0}K_0(H_n(0)).
\]
These two Grothendieck groups are dual Hopf algebras with product and coproduct again given by induction and restriction of representations. Krob and Thibon~\cite{KrobThibon} introduced Hopf algebra isomorphisms
\[
\mathrm{Ch}: G_0(H_\bullet(0))\cong\QSym,\quad \mathbf{ch}: K_0(H_\bullet(0))\cong \NSym
\]
which we review next.

Let $M=M_0\supseteq M_1\supseteq \cdots \supseteq M_k \supseteq M_{k+1}=0$ be a composition series of $H_n(0)$-modules with simple factors $M_i/M_{i+1} \cong \C_{\alpha^{(i)}}$ for $i=0,1,\ldots,k$. Then the \emph{quasisymmetric characteristic} of $M$ is
\[ 
\mathrm{Ch}(M):=F_{\alpha^{(0)}}+\cdots+F_{\alpha^{(k)}}. 
\]
This is well defined by the Jordan-H\"older theorem. The \emph{noncommutative characteristic} of a projective $H_n(0)$-module $M \cong \P_{\alpha^{(1)}}\oplus\cdots\oplus\P_{\alpha^{(k)}}$ is
\[
\mathbf{ch}(M):=\bs_{\alpha^{(1)}}+\cdots+\bs_{\alpha^{(k)}}.
\]
Krob and Thibon~\cite{KrobThibon} also defined the \emph{length-graded quasisymmetric characteristic} of a cyclic $H_n(0)$-module $N=H_n(0)v$ to be 
\[
\mathrm{Ch}_q(N):=\sum_{\ell\geq0} q^\ell \mathrm{Ch}(N^{(\ell)}/N^{(\ell+1)})
\]
where $N^{(\ell)}$ is the $H_n(0)$-submodule of $N$ spanned by $\{\pib_w v: w\in \SS_n,\ \ell(w)\geq \ell\}$. For example (c.f. Lemma~\ref{lem:Na}), if $\alpha$ is a composition of $n$, then the cyclic $H_n(0)$-module $H_n(0)\pi_{w_0(\alpha^c)}$ has length-graded quasisymmetric characteristic
\[
\mathrm{Ch}_q(H_n(0)\pi_{w_0(\alpha^c)}) = \sum_{w\in\SS^\alpha} q^{\inv(w)} F_{D(w^{-1})}.
\]

We often consider multigraded $H_n(0)$-modules $M$ with countably many homogeneous components that are all finite dimensional. Since each component of $M$ has a well-defined quasisymmetric characteristic and a multigrading, we obtain a multigraded quasisymmetric characteristic of $M$, which can be combined with the aforementioned length-graded quasisymmetric characteristic if in addition every homogeneous component is cyclic. If $M$ is projective then one has a multigraded noncommutative characteristic of $M$. The multigraded Frobenius characteristic is defined in the same way for multigraded $\CC\SS_n$-modules.

%Krob and Thibon \cite{KrobThibon} also showed that $\mathrm{Ch}(\P_\alpha)=s_\alpha$ is the ribbon Schur function, which is the commutative image of $\mathbf{ch}(P_\alpha)=\bs_\alpha$. Thus $\mathrm{Ch}(M)$ is symmetric if $M$ is a finitely generated projective $H_n(0)$-module, but not vice versa (e.g. $\C_{12}\oplus\C_{21}$ is nonprojective but $\mathrm{Ch}(\C_{12}\oplus\C_{21})=s_{21}\in\Sym$).
%If each direct summand $\P_{\alpha^{(i)}}$ has a homogeneous multigrading $t_1^{d_{i,1}} \cdots t_m^{d_{i,m}}$ then we define the multigraded noncommutative characteristic of $M$ to be
%\[
%\mathbf{ch}_{t_1,\ldots,t_m} (M):=\sum_{i=1}^k  t_1^{d_{i,1}} \cdots t_m^{d_{i,m}} %\bs_{\alpha^{(i)}}.
%\]

\section{Stanley-Reisner ring of the Boolean algebra}\label{sec:SRBoolean}

In this section we study the Stanley-Reisner ring of the Boolean algebra.

\subsection{Definition of the Stanley-Reisner ring}
An \emph{(abstract) simplicial complex} $\Delta$ on a vertex set $V$ is a collection of finite subsets of $V$, called \emph{faces}, such that any subset of a face is also a face. The \emph{dimension of a face} $F$ is $|F|-1$, and the \emph{dimension of a simplicial complex} is the maximum dimension of its faces. A $(d-1)$-dimensional simplicial complex is \emph{balanced} if there exists a coloring map $r:V\to[d]$ such that every face consists of vertices of distinct colors. The \emph{rank set} of a face $F$, denote by $r(F)$, consists of the colors of all its vertices.

The \emph{Stanley-Reisner ring} $\FF[\Delta]$ of a simplicial complex $\Delta$ over a field $\FF$ is defined as the quotient of the polynomial ring $\FF[y_v:v\in V]$ by its ideal 
\[
I_\Delta:=\left(y_u y_v :u,v\in V,\ \{u,v\}\notin\Delta\right).
\]
A monomial $y_{v_1}\cdots y_{v_k}$ is nonzero if and only if $v_1,\ldots v_k$ belong to the same face of $\Delta$. All nonzero monomials form an $\FF$-basis for $\FF[\Delta]$.

If $\Delta$ is balanced then its Stanley-Reisner ring $\FF[\Delta]$ is multigraded: any nonzero monomial $m=y_{v_1}\cdots y_{v_k}$ has \emph{rank multiset} $r(m)$ equal to the multiset of the ranks $r(v_1),\ldots,r(v_k)$, and receives the multigrading $t_{r(v_1)}\cdots t_{r(v_k)}$. The \emph{rank polynomials} in $\FF[\Delta]$ are the homogeneous elements
\[
\theta_i:=\sum_{r(v)=i} y_v,\quad i=1,\ldots,d.
\]
One sees that $\theta_1^{a_1}\cdots\theta_d^{a_d}$ equals the sum of all nonzero monomials $m$ with $r(m)=\{1^{a_1},\ldots,d^{a_d}\}$. Hence $\FF[\Theta]$ is a polynomial subalgebra of $\FF[\Delta]$, where $\Theta:=\{\theta_1,\ldots,\theta_d\}$. 

The Stanley-Reisner ring $\FF[P]$ of a finite poset $P$ is the same as the Stanley-Reisner ring of its \emph{order complex}, whose faces are the chains in $P$ ordered by reverse refinement. Explicitly,
\[
\FF[P]:=\FF[y_v:v\in P]/(y_uy_v:\textrm{ $u$ and $v$ are incomparable in $P$}).
\]
The nonzero monomials in $\FF[P]$ are indexed by multichains of $P$. Multiplying nonzero monomials corresponds to merging the indexing multichains; the result is zero if the multichains cannot be merged together. If $P$ is ranked then its order complex is balanced and its Stanley-Reisner ring $\FF[P]$ is multigraded.	

\subsection{Boolean algebra}\label{sec:Boolean}

The Boolean algebra $\cB_n$ is the ranked poset of all subsets of $[n]$ ordered by inclusion, with minimum element $\emptyset$ and maximum element $[n]$. The rank of a subset of $[n]$ is defined as its cardinality. By definition given earlier, the Stanley-Reisner ring $\FF[\cB_n]$ of the Boolean algebra $\cB_n$ is the quotient of the polynomial algebra $\FF\left[y_{\ds A}:A\subseteq [n] \right]$ by the ideal $\left( y_{\ds A} y_{\ds B}: \textrm{ $A$ and $B$ are incomparable in $\cB_n$} \right)$. It has an $\FF$-basis $\{y_\MM\}$ indexed by the multichains $\MM$ in $\cB_n$, and is multigraded by the rank multisets $r(\MM)$ of the multichains $\MM$. 

The symmetric group $\SS_n$ acts on the Boolean algebra $\cB_n$ by permuting the integers $1,\ldots,n$. This induces an $\SS_n$-action on the Stanley-Reisner ring $\FF[\cB_n]$, preserving its multigrading.  The \emph{invariant algebra} $\FF[\cB_n]^{\SS_n}$ consists of all elements in $\FF[\cB_n]$ invariant under this $\SS_n$-action. For $i=0,1,\ldots,n$, the rank polynomial $\theta_i:=\sum_{|A|=i} y_A$ is obviously invariant under the $\SS_n$-action; the converse is also true.

\begin{proposition}\label{prop:W-invariants}
The invariant algebra $\FF[\cB_n]^{\SS_n}$ equals $\FF[\Theta]$, where $\Theta:=\{\theta_0,\ldots,\theta_n\}$.
\end{proposition}

\begin{proof}
It suffices to show $\FF[\cB_n]^{\SS_n} \subseteq \FF[\Theta]$. The $\SS_n$-action on $\FF[\cB_n]$ breaks up the set of nonzero monomials into orbits, and the orbit sums form an $\FF$-basis for $\FF[\cB_n]^{\SS_n}$. The $\SS_n$-orbit of a nonzero monomial with rank multiset $\{0^{a_0},\ldots,n^{a_n}\}$ consists of all nonzero monomials with the same rank multiset, and the corresponding orbit sum equals $\theta_0^{a_0}\cdots\theta_n^{a_n}$. This completes the proof.
\end{proof}

Garsia~\cite{Garsia} showed that $\FF[\cB_n]$ is a free $\FF[\Theta]$-module on the basis of descent monomials 
\[
Y_w:=\prod_{i\in D(w)} y_{\{w(1),\ldots,w(i)\}}, \quad \forall w\in\SS_n.
\]

\subsection{Multichains in $\cB_n$}\label{sec:multichains}

We study the multichains in $\cB_n$, as they naturally index an $\FF$-basis for $\FF[\cB_n]$. We first introduce some notation. A \emph{weak composition} with length $k\geq0$ is a sequence $\alpha=(\alpha_1,\ldots,\alpha_k)$ of $k$ \emph{nonnegative} (positive for a composition) integers. The \emph{size} of $\alpha$ is $|\alpha|:=\alpha_1+\cdots+\alpha_n$. If $|\alpha|=n$ then we say $\alpha$ is a weak composition of $n$. We denote by $\Com(n,k)$ the set of all weak compositions of $n$ with length $k$. The \emph{descent multiset} of $\alpha$ is the \emph{multiset} $D(\alpha):=\{\alpha_1,\alpha_1+\alpha_2,\alpha_1+\cdots+\alpha_{k-1}\}$. The map $\alpha\mapsto D(\alpha)$ gives a bijection between weak compositions of $n$ and multisets with elements in $\{0,\ldots,n\}$. The parabolic subgroup $\SS_\alpha$ is the same as the parabolic subgroup of $\SS_n$ indexed by the underlying composition of $\alpha$ obtained by removing all zeros from $\alpha$; similarly for $\SS^\alpha$. 

The homogeneous components of $\FF[\cB_n]$ are indexed by multisets with elements in $\{0,\ldots,n\}$, or equivalently by weak compositions $\alpha$ of $n$. The $\alpha$-homogeneous component $\FF[\cB_n]_\alpha$ has an $\FF$-basis  $\{y_\MM:r(\MM)=D(\alpha)\}$.

Let $\MM=(A_1\subseteq\cdots\subseteq A_k)$ be an arbitrary multichain of length $k$ in $\cB_n$; set $A_0:=\emptyset$ and $A_{k+1}:=[n]$ by convention. Define $\alpha(\MM):=(\alpha_1,\ldots,\alpha_{k+1})$, where $\alpha_i=|A_i|-|A_{i-1}|$ for all $i\in[k+1]$. Then $\alpha(\MM)\in \Com(n,k+1)$ and $D(\alpha(\MM))=r(\MM)$, i.e. $\alpha(\MM)$ indexes the homogeneous component containing $y_\MM$. Define $\sigma(\MM)$ to be the minimal element in $\SS_n$ which sends the standard multichain $[\alpha_1]\subseteq[\alpha_1+\alpha_2]\subseteq\cdots\subseteq [\alpha_1+\cdots+\alpha_k]$ with rank multiset $D(\alpha(\MM))$ to $\MM$. Then $\sigma(\MM)\in\SS^{\alpha(\MM)}$. 

The map $\MM\mapsto(\alpha(\MM),\sigma(\MM))$ is a bijection between multichains of length $k$ in $\cB_n$ and the pairs $(\alpha,\sigma)$ of $\alpha\in\Com(n,k+1)$ and $\sigma\in\SS^\alpha$. A short way to write down this encoding of $\MM$ is to insert bars at the descent positions of $\sigma(\MM)$. For example, the length-$4$ multichain $\{2\}\subseteq \{2\} \subseteq \{1,2,4\} \subseteq [4]$ in $\cB_4$ is encoded by $2||14|3|$. 

There is another way to encode the multichain $\MM$. Let $p_i(\MM):=\min\{j\in[k+1]: i\in A_j\}$, i.e. the first position where $i$ appears in $\MM$, for $i=1,\ldots,n$. One checks that
\begin{equation} \label{eq:equiv}
\begin{cases}
p_i(\MM)>p_{i+1}(\MM) \Leftrightarrow i\in D(\sigma(\MM)^{-1}), \\
p_i(\MM)=p_{i+1}(\MM) \Leftrightarrow i\notin D(\sigma(\MM)^{-1}),\ D(s_i\sigma(\MM))\not\subseteq D(\alpha(\MM)), \\
p_i(\MM)<p_{i+1}(\MM) \Leftrightarrow i\notin D(\sigma(\MM)^{-1}),\ D(s_i\sigma(\MM))\subseteq D(\alpha(\MM)).
\end{cases}
\end{equation}
This will be used later when we study the $H_n(0)$-action on $\FF[\cB_n]$. We define $p(\MM):=(p_1(\MM),\ldots,p_n(\MM))$. Then $\MM\mapsto p(\MM)$ gives an bijection between the set of multichains with length $k$ in $\cB_n$ and the set $[k+1]^n$ of all words of length $n$ on the alphabet $[k+1]$, for any fixed integer $k\geq0$. 

Suppose that $p(\MM)=\pp=(p_1,\ldots,p_n)\in[k+1]^n$. We define $\inv(\pp):=\#\{(i,j): 1\leq i<j\leq n,\  p_i> p_j\}.$ One sees that $\inv(\sigma(\MM))=\inv(p(\MM))$. Let $\pp':=(p'_1,\ldots,p'_k)$ where $p'_i:=|\{j:p_j(\MM)\leq i\}|=|A_i|$. Then the rank multiset of $\MM$ consists of $p'_1,\ldots,p'_k$. If we draw $k+1-p_j$ boxes on the $j$-th row of a $n\times k$ rectangle for all $j\in[n]$, then $p'_i$ is the number of boxes on the $(k+1-i)$-th column. For example, the multichain $3|14||2|5$ corresponds to $\pp=(2,4,1,2,5)\in[5]^5$, and one has $\pp'=(1,3,3,4)$ and the following picture.
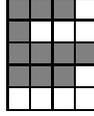
\[
\ydiagram
[*(gray)]{3,1,4,3,0}
*[*(white)]{4,4,4,4,4}
\]
This implies an equation which will be used later:
\begin{equation}\label{eq:box}
(q^k+\cdots+q+1)^n = \sum_{\pp\in[k+1]^n}\prod_{1\leq j\leq n} q^{k+1-p_j}
= \sum_{\pp\in[k+1]^n}\prod_{1\leq i\leq k} q^{p'_i}.
\end{equation}
Define $
D(\pp):=\{i\in[n-1]:p_i>p_{i+1}\}.
$
For example, $D(2,5,1,2,4)=\{2\}$.

These two encodings (with slightly different notation) were already used by Garsia and Gessel~\cite{GarsiaGessel} in their work on generating functions of multivariate distributions of permutation statistics. %$\inv(w)$, $\maj(w)$, $\maj(w^{-1})$, $\des(w)$, $\des(w^{-1})$ for all $w\in\SS_n$. 
In the next section we will use these encodings to derive multivariate quasisymmetric function identities from $H_n(0)$-action on the Stanley-Reisner ring of $\cB_n$, giving generalizations of some results of Garsia and Gessel~\cite{GarsiaGessel}.

\subsection{Rank-selection}\label{sec:rank}

Let $\alpha$ be a composition of $n$. We define the \emph{rank-selected Boolean algebra} 
\[
\cB_\alpha:=\{A\subseteq [n]: |A|\in D(\alpha)\cup\{n\}\}
\]
which is a ranked subposet of the Boolean algebra $\cB_n$. We always exclude $\emptyset$ but keep $[n]$ because there is a nice analogy between the Stanley-Reisner ring of $\cB_n^* := \cB_{1^n} = \cB_n\setminus\{\emptyset\}$ and the polynomial ring $\FF[\px]$. 

To explain this analogy, we use the \emph{transfer map} $\tau:\FF[\cB_n]\to \FF[X]$ defined by
\[
\tau(y_\MM):= \prod_{1\leq i\leq k} \prod_{j\in A_i} x_j
\]
for all multichains $\MM=(A_1\subseteq\cdots\subseteq A_k)$ in $\cB_n$. It is \emph{not} a ring homomorphism (e.g. $y_{\{1\}}y_{\{2\}}=0$ but $x_1x_2\ne0$). Nevertheless, it restricts to an isomorphism $\tau:\FF[\cB_n^*]\cong\FF[X]$ of $\SS_n$-modules. 

The invariant algebra $\FF[\cB_n^*]^{\SS_n}$ equals the polynomial algebra $\FF[\theta_1,\ldots,\theta_n]$, and $\FF[\cB_n^*]$ is a free $\FF[\theta_1,\ldots,\theta_n]$-module on the descent basis $\{Y_w:w\in\SS_n\}$. The transfer map $\tau$ sends the rank polynomials $\theta_1,\ldots,\theta_n$ to the elementary symmetric polynomials $e_1,\ldots,e_n$, which generate the invariant algebra $\FF[X]^{\SS_n}$. It also sends the descent monomials $Y_w$ in $\FF[\cB_n^*]$ to the descent monomials 
\[
X_w:= \prod_{i\in D(w)} x_{w(1)}\cdots x_{w(i)} 
\]
in $\FF[X]$ for all $w\in\SS_n$, which form a free $\FF[X]^{\SS_n}$-basis for $\FF[X]$ (see e.g. Garsia~\cite{Garsia}). Therefore $\FF[\cB_n^*]$ is in a nice analogy with $\FF[X]$ via the transfer map $\tau$. 

\begin{remark}
The Stanley-Reisner ring $\FF[\cB_n]$ is not much different from $\FF[\cB_n^*]$, as one can see the $\FF$-algebra isomorphisms $\FF[\cB_n] \cong \FF[\cB_n^*] \otimes_\FF \FF[\theta_0]$ and $\FF[\cB_n]/(\theta_0)\cong\FF[\cB_n^*]$, where $\theta_0=y_\emptyset$.
\end{remark}

In general, the Stanley-Reisner ring of the rank-selected Boolean algebra $\cB_\alpha$ is the multigraded subalgebra of $\FF[\cB_n]$ generated by $\{y_{\ds A}:|A|\in D(\alpha)\cup\{n\}\}$. There is also a projection $\phi: \FF[\cB_n] \twoheadrightarrow \FF[\cB_\alpha]$ of multigraded algebras given by
\[
\phi(y_{\ds A}):= \begin{cases}
y_{\ds A}, & {\rm if}\ A\subseteq [n],\ |A|\in D(\alpha)\cup\{n\},\\
0, & {\rm if}\ A\subseteq[n],\ |A|\notin D(\alpha)\cup\{n\}.
\end{cases}
\]
The $\SS_n$-action preserves both the inclusion $\FF[\cB_\alpha]\subseteq \FF[\cB_n]$ and the projection $\phi$. Thus one has an isomorphism
\[
\FF[\cB_n]/( A\subseteq [n]: |A|\notin D(\alpha) \cup\{n\} ) \cong \FF[\cB_\alpha]
\] 
of multigraded $\FF$-algebras and $\SS_n$-modules.

Applying the projection $\phi$ one sees that the invariant algebra $\FF[\cB_\alpha]^{\SS_n}$ is the polynomial algebra $\FF[\Theta_\alpha]$, where $\Theta_\alpha:=\{\theta_i:i\in D(\alpha)\cup\{n\} \}$, and $\FF[\cB_\alpha]$ is a free $\FF[\Theta_\alpha]$-module on the basis of descent monomials $Y_w$ for all $w\in\SS^\alpha$.

\section{$0$-Hecke algebra action on the Stanley-Reisner ring of the Boolean algebra}\label{sec:H0SR}

In this section we define an action of the $0$-Hecke algebra $H_n(0)$ on the Stanley-Reisner ring $[\cB_n]$ and establish our main results Theorem~\ref{thm1} and Theorem~\ref{thm2}.

\subsection{Definition of the $0$-Hecke algebra action} \label{sec:H0Bn}

We saw an analogy between $\FF[\cB_n]$ and $\FF[X]$ in the last section. The usual $H_n(0)$-action on the polynomial ring $\FF[\px]$ is via the \emph{Demazure operators}
\begin{equation}\label{eq:Demazure}
\overline\pi_i(f):=\frac{x_{i+1}f-x_{i+1}s_if}{x_i-x_{i+1}},\quad \forall f\in \FF[\px],\ 1\leq i\leq n-1.
\end{equation}
The above definition is equivalent to 
\begin{equation}\label{eq:pim}
\overline\pi_i(x_i^ax_{i+1}^bm)=\left\{\begin{array}{ll} (x_i^{a-1}x_{i+1}^{b+1}+x_i^{a-2}x_{i+1}^{b+2}\cdots +x_i^bx_{i+1}^a)m, &{\rm if}\ a>b,\\
0, &{\rm if}\ a=b,\\
-(x_i^ax_{i+1}^b+x_i^{a+1}x_{i+1}^{b-1}+\cdots+x_i^{b-1}x_{i+1}^{a+1})m, &{\rm if}\ a<b.
\end{array}\right.
\end{equation}
Here $m$ is a monomial in $\FF[\px]$ containing neither $x_i$ nor $x_{i+1}$. Denote by $\pib'_i$ the operator obtained from (\ref{eq:pim}) by taking only the leading term (under the lexicographic order). Then $\pib'_1,\ldots,\pib'_{n-1}$ realize another $H_n(0)$-action on $\FF[X]$. We call it the \emph{transferred $H_n(0)$-action} because it can be obtained by applying the transfer map $\tau$ to our $H_n(0)$-action on $\FF[\cB_n]$, which we now define.

Let $\MM=(A_1\subseteq\cdots\subseteq A_k)$ be a multichain in $\cB_n$. Recall $p_i(\MM):=\min\{j\in[k+1]: i\in A_j\}$. We define  
\begin{equation}\label{def:H0Boolean}
\overline\pi_i(y_\MM):=
\begin{cases}
-y_\MM, & p_i(\MM)>p_{i+1}(\MM), \\
0, & p_i(\MM)=p_{i+1}(\MM), \\
s_i(y_\MM), & p_i(\MM)<p_{i+1}(\MM)
\end{cases}
\end{equation} 
for $i=1,\ldots,n-1$. Applying the transfer map $\tau$ one recovers $\pib'_i$. For instance, when $n=4$ one has
\[
\overline\pi_1(y_{1|34||2|})=y_{2|34||1|},\quad \pib_2(y_{1|34||2|})=-y_{1|34||2|},\quad  \overline\pi_3(y_{1|34||2|})=0.
\]
Applying the transfer map one has
\[
\pib'_1(x_1^4x_2x_3^3x_4^3)=x_1x_2^4x_3^3x_4^3,\quad \pib'_2(x_1^4x_2x_3^3x_4^3)=-x_1^4x_2x_3^3x_4^3,\quad  \pib'_3(x_1^4x_2x_3^3x_4^3)=0.
\]

It is easy to see that $\pib_i^2 = -\pib_i$ and $\pib_i\pib_j=\pib_j\pib_i$ whenever $1\leq i,j\leq n-1$ and $|i-j|>1$. For any $i\in[n-2]$, one can consider different possibilities for the relative positions of $p_i(\MM)$, $p_{i+1}(\MM)$, and $p_{i+2}(\MM)$, and verify the relation $\pib_i\pib_{i+1}\pib_i=\pib_{i+1}\pib_i\pib_{i+1}$ case by case. Hence $H_n(0)$ acts on $\FF[\cB_n]$ via the above defined operators $\pib_1,\ldots,\pib_{n-1}$. This $H_n(0)$-action  preserves the multigrading of $\FF[\cB_n]$, and thus restricts to the Stanley-Reisner ring $\FF[\cB_\alpha]$ for any composition $\alpha$ of $n$.

Another way to see that $\pib_1,\ldots,\pib_{n-1}$ realize an $H_n(0)$-action on $\FF[\cB_n]$  preserving the multigrading is to show that each homogeneous component of $\FF[\cB_n]$ is isomorphic to an $H_n(0)$-module. We will give this in Lemma~\ref{lem:Na}.

\begin{remark} 
Our $H_n(0)$ action on $\FF[\cB_n]$ also has a similar expression to (\ref{eq:Demazure}) as one can show that it has the following properties.

\vskip5pt\noindent (i)
If $f,g,h$ are elements in $\FF[\cB_n]$ such that $f=gh$ and $h$ is homogeneous, then there exists a unique element $g'$, defined as the quotient $f/h$, such that $f=g'h$ and $y_\MM h \ne 0$ for every monomial $y_\MM$ appearing in $g'$.

\vskip5pt\noindent (ii) 
Suppose that $i\in[n-1]$ and $\MM=(A_1\subseteq\cdots\subseteq A_k)$ is a multichain in $\cB_n$. Let $j$ be the largest integer in $\{0,\ldots,k\}$ such that $A_j\cap\{i,i+1\}=\emptyset$. Then $s_i(y_\MM) = y_{\ds M^{(i)}} s_i(y_{\MM_i})$ and $\overline\pi_i(y_\MM) = y_{\ds M^{(i)}} \overline\pi_i(y_{\MM_i})$, where $\MM^{(i)}: A_1 \subseteq \cdots \subseteq A_j$ and $\MM_i: A_{j+1} \subseteq \cdots \subseteq A_k$.

\vskip5pt\noindent(iii)
One has
\[
\overline\pi_i(y_{\MM_i})=\frac{ y_{\ds\{i+1\}} y_{\MM_i} - y_{\ds\{i+1\}} s_i(y_{\MM_i}) } {y_{\ds\{i\}}-y_{\ds\{i+1\}}}.
\]
\end{remark}

\subsection{Basic properties}

We define the \emph{invariant algebra} $\FF[\cB_n]^{H_n(0)}$ of the $H_n(0)$-action on $\FF[\cB_n]$ to be the trivial isotypic component of $\FF[\cB_n]$ as an $H_n(0)$-module, namely
\begin{eqnarray*}
\FF[\cB_n]^{H_n(0)} &:=& \left\{ f\in \FF[\cB_n]: \pib_i f = 0,\ i=1,\ldots,n-1 \right\} \\
&=& \left\{ f\in \FF[\cB_n]: \pi_i f = f,\ i=1,\ldots,n-1 \right\}.
\end{eqnarray*}
This is an analogue of the invariant algebra $\FF[\cB_n]^{\SS_n}$, which equals $\FF[\Theta]$ by Proposition~\ref{prop:W-invariants}, and we show that they are actually the same. 

\begin{proposition}
The invariant algebra $\FF[\cB_n]^{H_n(0)}$ equals $\FF[\Theta]$.
\end{proposition}

\begin{proof}
Let $i\in[n-1]$. We denote by $\mathcal M_1$, $\mathcal M_2$, and $\mathcal M_3$ the sets of all multichains $\MM$ in $\cB_n$ with $p_i(\MM)<p_{i+1}(\MM)$, $p_i(\MM)=p_{i+1}(\MM)$, and $p_i(\MM)>p_{i+1}(\MM)$, respectively.
The action of $s_i$ pointwise fixes $\mathcal M_2$ and bijectively sends $\mathcal M_1$ to $\mathcal M_3$. It follows from (\ref{def:H0Boolean}) that $\pib_i(f)=0$ if and only if 
\[
f=\sum_{\MM\in \mathcal M_1} a_\MM (y_\MM+y_{s_i \MM}) +\sum_{\MM\in \mathcal M_2} b_\MM y_\MM,\quad a,b\in\FF.
\]
This is also equivalent to $s_i(f)=f$. Therefore $\FF[\cB_n]^{H_n(0)}=\FF[\cB_n]^{\SS_n}=\FF[\Theta]$.
\end{proof}

The $\SS_n$-action on $\FF[\cB_n]$ is $\Theta$-linear, and so is the $H_n(0)$-action. 

\begin{proposition}\label{prop:linear}
The $H_n(0)$-action on $\FF[\cB_n]$ is $\Theta$-linear.
\end{proposition}

\begin{proof}
Let $i\in[n-1]$ and let $\MM=(A_1\subseteq\cdots\subseteq A_k)$ be an arbitrary multichain in $\cB_n$. Since $\theta_0=\emptyset$, one has $\pib_i(\theta_0 y_\MM) = \theta_0 \pib_i(y_\MM)$. It remains to show $\pib_i(\theta_r y_{\MM})=\theta_r \pib_i(y_{\MM})$ for any $r\in[n]$. One has $|A_j|< r \leq|A_{j+1}|$ for some $j\in\{0,1,\ldots,k\}$, where $A_0=\emptyset$ and $A_{k+1}=[n]$ by convention. Then
\[
\theta_r y_\MM=\sum_{A\in \mathcal A} y_\MM y_{\ds A}.
\]
where 
\[
\mathcal A:=\{A\subseteq[n]: |A|=r,\ A_j\subsetneq A\subseteq A_{j+1} \}.
\]
If $A\in\mathcal A$ then $y_\MM y_{\ds A} = y_{\MM\cup A}$ where $\MM\cup A$ is the multichain obtained by inserting $A$ into $\MM$. Let $\mathcal A_1$ [ $\mathcal A_2$, $\mathcal A_3$, resp. ] be the collections of all sets $A$ in $\mathcal A$ satisfying 
\[
p_i(\MM\cup A) < \textrm{ [$=$, $>$, resp.] } p_{i+1}(\MM\cup A)
\]
We distinguish three cases below.

If $p_i(\MM)>p_{i+1}(\MM)$, then $\pib_i(y_\MM)=-y_\MM$. Assume $A\in \mathcal A$. One has $p_i(\MM\cup A)>p_{i+1}(\MM\cup A)$ when $i\notin A$. If $i\in A \subseteq A_{j+1}$ then $p_i(\MM)>p_{i+1}(\MM)$ forces $i+1\in A_j \subseteq A$ and one still has $p_i(\MM\cup A) > p_{i+1}(\MM\cup A)$. Hence $\mathcal A = \mathcal A_3$ which implies $\pib_i(\theta_r y_\MM) = - \theta_r y_\MM = \theta_r \pib_i(y_\MM)$.

If $p_i(\MM)=p_{i+1}(\MM)$, then $\pib_i(y_\MM)=0$ and we need to show $\pib_i(\theta_r y_\MM)=0$. First assume $A\in\mathcal A_1$. Then $A$ contains $i$ but not $i+1$, $A_j$ contains neither, and $A_{j+1}$ contains both. Hence $s_i(A)\in\mathcal A_3$ and 
\[
\pib_i(y_{\MM\cup A}) = s_i(y_{\MM\cup A}) = y_{s_i(\MM\cup A)}.
\]
Similarly if $A\in\mathcal A_3$ then $s_i(A)\in\mathcal A_1$ and thus $s_i$ gives an bijection between $\mathcal A_1$ and $\mathcal A_3$. For any $A\in\mathcal A_2$ one has $\pi_i(y_{\MM\cup A})=0$. Therefore 
\[
\pib_i(\theta_r y_\MM) = \sum_{A\in\mathcal A_1} \pib_i(y_{\MM\cup A} + \pib_i(y_{\MM\cup A}) ) = 0.
\]
Here the last equality follows from the relation $\pib_i^2=-\pib_i$. 

Finally, we consider the case $p_i(\MM)<p_{i+1}(\MM)$. Assume $A\in\mathcal A$. One has $p_i(\MM\cup A)<p_{i+1}(\MM\cup A)$ when $i\in A_j$. If $i\notin A_j$, then $i+1\notin A_{j+1}$ and so $i+1\notin A$, which implies $p_i(\MM\cup A)<p_{i+1}(\MM\cup A)$. Thus $\mathcal A = \mathcal A_1$ and $\pib_i(\theta_r y_\MM) = s_i(\theta_r y_\MM) = \theta_r s_i(y_\MM) = \theta_r \overline\pi_i(y_\MM)$. 
\end{proof}

Therefore the \emph{coinvariant algebra} $\FF[\cB_n]/(\Theta)$ is a multigraded $H_n(0)$-module, and we will see in the next subsection that it carries the regular representation of $H_n(0)$. This cannot be obtained simply by applying the transfer map $\tau$, since $\tau$ is \emph{not} a map of $H_n(0)$-modules (see \S\ref{sec:transfer}).

\subsection{Noncommutative Hall-Littlewood symmetric functions}\label{sec:NH}
In this subsection we write a partition of $n$ as an increasing sequence $\mu=(0<\mu_1\leq\cdots\leq\mu_k)$ of positive integers whose sum is $n$, and view it as a composition in this way whenever needed. We want to establish a noncommutative analogue of the following remarkable result. 

\begin{theorem}[Hotta-Springer~\cite{HottaSpringer}, Garsia-Procesi~\cite{GarsiaProcesi}]\label{thm:Springer}
For every partition $\mu=(0<\mu_1\leq\cdots\leq\mu_k)$ of $n$, there exists an $\SS_n$-invariant ideal $J_\mu$ of $\CC[\px]$ such that $\CC[\px]/J_\mu$ is isomorphic to the cohomology ring of the Springer fiber indexed by $\mu$ and has graded Frobenius characteristic equal to the modified Hall-Littlewood symmetric function 
\[
\widetilde H_\mu(X;t)= \sum_\lambda t^{n(\mu)} K_{\lambda\mu}(t^{-1})s_\lambda \quad {\rm inside}\quad \Sym[t]
\]
where $K_{\lambda\mu}(t)$ is the Kostka-Foulkes polynomial and $n(\mu):=\mu_{k-1}+2\mu_{k-2}+\cdots+(k-1)\mu_1$.
\end{theorem}

By Tanisaki~\cite{Tanisaki}, the ideal $J_\mu$ is generated by 
\[
\{e_i(S): |S| \geq i> |S|-(\mu'_1+\cdots+\mu'_{|S|}),\ S\subseteq\{x_1,\ldots,x_n\}\}
\]
where $\mu'=(0\leq\mu'_1\leq\cdots\leq\mu'_n)$ is the conjugate of the partition $\mu$ with zero parts added whenever necessary, and $e_i(S)$ is the $i$-th elementary symmetric function in the set $S$ of variables. See also Garsia and Procesi~\cite{GarsiaProcesi}. We can work over an arbitrary field $\FF$, and still denote by $J_\mu$ the ideal of $\FF[\px]$ with the same generators. For instance, if $\mu=(1^k,n-k)$ is a hook, then $\mu'=(0^k,1^{n-k-1},k+1)$, and one can check that the ideal $J_{1^k,n-k}$ is generated by $e_1,\ldots,e_k$ and all monomials $x_{i_1}\cdots x_{i_{k+1}}$ with $1\leq i_1<\cdots <i_{k+1}$.

Now we consider an arbitrary composition $\alpha=(\alpha_1,\ldots,\alpha_\ell)$. The major index of $\alpha$ is $\maj(\alpha):=\sum_{i\in D(\alpha)}i$, and viewing a partition $\mu$ as a composition one has $\maj(\mu)=n(\mu)$. Recall that $\alpha^c$ is the composition of $n$ with $D(\alpha^c)=[n-1]\setminus D(\alpha)$ and $\overleftarrow\alpha:=(\alpha_\ell,\ldots,\alpha_1)$. We define  $\alpha':= \overleftarrow{\alpha^c} = (\overleftarrow{\alpha})^c$. One can identify $\alpha$ with a \emph{ribbon diagram}, i.e. a connected skew Young diagram without $2$ by $2$ boxes, which has row lengths $\alpha_1,\ldots,\alpha_\ell$, ordered from bottom to top. Note that a ribbon diagram is a Young diagram if and only if it is a hook. One can check that $\alpha'$ is the transpose of $\alpha$; see the example below.
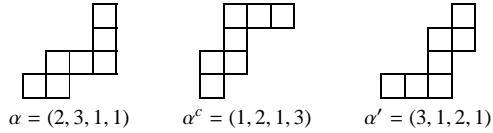
\[
\footnotesize
\begin{array}{ccccc}
\young(:::\hfill,:::\hfill,:\hfill\hfill\hfill,\hfill\hfill) &&
\young(:\hfill\hfill\hfill,:\hfill,\hfill\hfill,\hfill) &&
\young(:::\hfill,::\hfill\hfill,::\hfill,\hfill\hfill\hfill) \\
\alpha=(2,3,1,1) & & \alpha^c=(1,2,1,3) & & \alpha'=(3,1,2,1)
\end{array} 
\]

Bergeron and Zabrocki~\cite{BZ} introduced a noncommutative modified Hall-Littlewood symmetric function
\[
\widetilde{\mathbf{H}}_\alpha(\mathbf x;t):=\sum_{\beta\cleq\alpha} t^{\maj(\beta)} \bs_\beta \quad {\rm inside}\quad \NSym[t]
\]
and a $(q,t)$-analogue
\[
\tNH_\alpha(\bx;q,t):=\sum_{\beta\models n}t^{c(\alpha,\beta)}q^{c(\alpha',\overleftarrow{\beta})}\bs_\beta \quad {\rm inside} \quad \NSym[q,t]
\]
for every composition $\alpha$, where $c(\alpha,\beta):=\sum_{i\in D(\alpha)\cap D(\beta)} i$. In our earlier work~\cite{H} we provided a partial representation theoretic interpretation for $\widetilde{\mathbf{H}}_\alpha(\mathbf x;t)$ when $\mu=(1^k,n-k)$ is a hook, using the $H_n(0)$-action on the polynomial ring $\FF[\px]$ by the Demazure operators.

\begin{theorem}[\cite{H}]
The ideal $J_\mu$ of $\FF[\px]$ is $H_n(0)$-invariant if and only if $\mu=(1^{n-k},k)$ is a hook, and if that holds then $\FF[\px]/J_\mu$ becomes a graded projective $H_n(0)$-module with 
\[
\mathbf{ch}_t(\FF[\px]/J_\mu) = \widetilde{\mathbf{H}}_\mu(\mathbf x;t),
\]
\[
\mathrm{Ch}_t(\FF[\px]/J_\mu) = \widetilde{H}_\mu(\mathbf x;t).
\]
\end{theorem}

%So far one has not known any similar result for compositions that are not hooks.

Now we switch to the Stanley-Reisner ring $\FF[\cB_n]$ and provide a complete representation theoretic interpretation for $\widetilde{\mathbf{H}}_\alpha(\mathbf x;t)$ and $\tNH_\alpha(\bx;q,t)$. Recall that $\underline t^S$ means the product of $t_i$ for all elements $i$ in a multiset $S\subseteq[n-1]$, with repetitions included.

\begin{theorem}\label{thm:NHL}
Let $\alpha$ be a composition of $n$, and let $I_\alpha$ be the ideal of $\FF[\cB_n]$ generated by \[
\Theta_\alpha:=\{\theta_i:i\in D(\alpha)\cup\{n\}\} \quad \textrm{and}\quad \{ y_A\subseteq [n]: |A|\notin D(\alpha) \cup \{n\}\}.
\]
Then one has an isomorphism 
$
\FF[\cB_n]/I_\alpha \cong \FF[\cB_\alpha]/(\Theta_\alpha)
$
of multigraded $\FF$-algebras, $\SS_n$-modules, and $H_n(0)$-modules. In addition, the multigraded noncommutative characteristic of $\FF[\cB_n]/I_\alpha$ equals
\[
\tNH_\alpha(\bx;t_1,\ldots,t_{n-1}):= \sum_{\beta\cleq\alpha} \underline t^{D(\beta)} \bs_\beta  \quad {\rm inside} \quad \NSym[t_1,\ldots,t_{n-1}].
\]
\end{theorem}

\begin{proof}
In \S\ref{sec:rank} we defined a projection $\phi: \FF[\cB_n] \twoheadrightarrow \FF[\cB_\alpha]$ which induces an isomorphism 
\[
\FF[\cB_n]/( A\subseteq [n]: |A|\notin D(\alpha) \cup\{n\} ) \cong \FF[\cB_\alpha]
\] 
of multigraded $\FF$-algebras and $\SS_n$-modules. By definition, the $H_n(0)$-action also preserves $\phi$. Since the actions of $\SS_n$ and $H_n(0)$ are both $\Theta_\alpha$-linear, we can take a further quotient by the ideal generated by $\Theta_\alpha$ and obtain the desired isomorphism $\FF[\cB_n]/I_\alpha \cong \FF[\cB_\alpha]/(\Theta_\alpha)$ of multigraded $\FF$-algebras, $\SS_n$-modules, and $H_n(0)$-modules.

Since $\FF[\cB_\alpha]/(\Theta_\alpha)$ has an $\FF$-basis of the descent monomials $Y_w$ for all $w\in\SS^\alpha$, it equals the direct sum of $Q_\beta$, the $\FF$-span of $\{Y_w:D(w)=D(\beta)\}$, for all $\beta\cleq\alpha$; each $Q_\beta$ has homogeneous multigrading $\underline t^{D(\beta)}$. The projective indecomposable $H_n(0)$-module $\P_\beta=H_n(0)\pib_{w_0(\beta)}\pi_{w_0(\beta^c)}$ has an $\FF$-basis 
\[
\left\{ \pib_w\pi_{w_0(\beta^c)}: D(w)=D(\beta) \right\}.
\] 
Thus one has an vector space isomorphism $Q_\beta \cong \P_\beta$ via $Y_w \mapsto \pib_w\pi_{w_0(\beta^c)}$. We wan to show that this isomorphism is $H_n(0)$-equivariant. Let $i\in[n-1]$ be arbitrary. Suppose that $D(w)=D(\beta)$, and let $\MM$ be the chain of the sets $\{w(1),\ldots,w(j)\}$ for all $j\in D(w)$. Then $\alpha(\MM)=\beta$ and $\sigma(\MM)=w$. We distinguish three cases below and use (\ref{eq:equiv}).

If $p_i(\MM)>p_{i+1}(\MM)$, i.e.  $i\in D(w^{-1})$, then $\pib_i(Y_w) = - Y_w$ and $\pib_i \pib_w\pi_{w_0(\beta^c)} = - \pib_w\pi_{w_0(\beta^c)}$.

If $p_i(\MM)=p_{i+1}(\MM)$, i.e. $i\notin D(w^{-1})$ and $D(s_iw)\not\subseteq D(\beta)$, then $\pib_i(Y_w) = 0$ and there exists $j\in D(s_iw)\setminus D(\beta)$ such that $\pib_i \pib_w\pi_{w_0(\beta^c)} = \pib_w \pib_j \pi_{w_0(\beta^ac)} =0$ since $\pib_j\pi_j=0$.

If $p_i(\MM)<p_{i+1}(\MM)$, i.e. $i\notin D(w^{-1})$ and $D(s_iw) \subseteq D(\beta)$, then $\pib_i(Y_w) = y_{s_iw}$ and $\pib_i \pib_w\pi_{w_0(\beta^c)} = \pib_{s_iw} \pi_{w_0(\beta^c)}$.

Therefore $Q_\beta\cong \P_\beta$ is an isomorphism of $H_n(0)$-modules for all $\beta\cleq\alpha$. It follows that the multigraded noncommutative characteristic of $\FF[\cB_n]/I_\alpha \cong \FF[\cB_\alpha]/(\Theta_\alpha)$ is $\tNH_\alpha(\bx;t_1,\ldots,t_{n-1})$.
\end{proof}

It is easy to see $\widetilde{\mathbf{H}}_\alpha(\mathbf x;t) = \tNH_\alpha(\mathbf{x};t,t^2,\ldots,t^{n-1})$. Thus Theorem~\ref{thm:NHL} provides a representation theoretic interpretation of $\widetilde{\mathbf{H}}_\alpha(\mathbf x;t)$ for all compositions $\alpha$, and can be viewed as a noncommutative analogue of Theorem~\ref{thm:Springer}.

\begin{remark}
The proof of Theorem~\ref{thm:NHL} is actually simpler than the proof of our partial interpretation for $\tNH_\alpha(\mathbf x;t)$ in \cite{H}. This is because $\overline\pi_i$ sends a descent monomial in $\FF[\cB_n]$ to either $0$ or $\pm1$ times a descent monomial, but sends a descent monomial in $\FF[\px]$ to a polynomial in general (whose leading term is still a descent monomial). We view the Stanley-Reisner ring $\FF[\cB_n]$ (or $\FF[\cB_n^*]\cong\FF[\cB_n]/(\emptyset)$) as a $0$-analogue of the polynomial ring $\FF[\px]$. For an odd (i.e. $q=-1$) analogue, see Lauda and Russell~\cite{LaudaRussell}.
\end{remark}

\begin{remark}
When $\alpha=(1^k,n-k)$ is a hook, one can check that the ideal $I_{1^k,n-k}$ of $\FF[\cB_n]$ has generators $\theta_1,\ldots,\theta_k$ and all $y_{\ds A}$ with $A\subseteq [n]$ and $|A|\notin[k]$. One can also check that the images of these generators under the transfer map $\tau$ are the Tanisaki generators for the ideal $J_{1^k,n-k}$ of $\FF[\px]$, although $\tau(I_{1^k,n-k})\ne J_{1^k,n-k}$.
\end{remark}

\begin{remark}
One sees that the coinvariant algebra $\FF[\cB_n]/(\Theta)$ carries the regular representation of $H_n(0)$, as its multigraded noncommutative characteristic equals 
\[
\tNH_{1^n}(\bx;t_1,\ldots,t_{n-1}):= \sum_{\beta\models n} \underline t^{D(\beta)} \bs_\beta.
\]
If we take $t_i=t^i$ for all $i\in D(\alpha)$, and $t_i=q^{n-i}$ for all $i\in D(\alpha^c)$ in $\tNH_{1^n}(\mathbf{x};t_1,\ldots,t_{n-1})$, then we obtain the $(q,t)$-analogue $\tNH_\alpha(\bx;q,t)$. % which was originally introduced by Bergeron and Zabrocki~\cite{BZ} using an analogue of the nabla operator.

Hivert, Lascoux, and Thibon~\cite{HivertLascouxThibon} defined a family of noncommutative symmetric functions on multiple parameters $q_i$ and $t_i$, which are similar to but different from the family $\left\{ \tNH_\alpha(\bx;q,t) \right\}$. A common generalization of these two families of noncommutative symmetric functions is discovered recently by Lascoux, Novelli, and Thibon~\cite{LascouxNovelliThibon}, that is, a family  $\left\{ P_\alpha \right\}$ of noncommutative symmetric functions having parameters associated with paths in binary trees.

In fact, one recovers $P_\alpha$ from $\tNH_{1^n}(\bx;t_1,\ldots,t_{n-1})$, the noncommutative characteristic of the coinvariant algebra $\FF[\cB_n]/(\Theta)$. For any composition $\alpha$ of $n$, let $u(\alpha)=u_1\cdots u_{n-1}$ be the Boolean word such that $u_i=1$ if $i\in D(\alpha)$ and $u_i=0$ otherwise. Let $y_{u_{1\ldots i}}$ be a parameter indexed by the Boolean word $u_1\cdots u_i$. It follows from the definition of $P_\alpha$~\cite[(31)]{LascouxNovelliThibon} that
\[
P_\alpha = \sum_{\beta\models n} \left( \prod_{i\in D(\beta)} y_{u_{1\ldots i}} \right) \bs_\beta.
\]
Then taking $t_i = y_{u_{1\ldots i}}$ one has $\tNH_{1^n}(\bx;y_{u_{1\ldots 1}},\ldots,y_{u_{1\ldots n-1}}) = P_\alpha$. For example, when $\alpha=211$ one has $u(\alpha)=011$ and 
$
\tNH_{1111}(\bx;y_0,y_{01},y_{011}) = \bs_4 + y_{011}\bs_{31} + y_{01}\bs_{22} + y_{01}y_{011}\bs_{211} + y_0\bs_{13} + y_0y_{011}\bs_{121} + y_0y_{01}\bs_{112} + y_0y_{01}y_{011}\bs_{1111} = P_{211}.
$
\end{remark}

\vskip5pt
The multigraded noncommutative characteristic $\tNH_\alpha(\bx;t_1,\ldots,t_{n-1})$, where $\alpha\models n$, is the modified version of 
\[
\NH_\alpha = \NH_\alpha(\bx;t_1,\ldots,t_{n-1}):=\sum_{\beta\cleq\alpha} \underline t^{D(\alpha)\setminus D(\beta)} \bs_\beta 
\]
which belongs to $\NSym[t_1,\ldots,t_{n-1}]$. We show below that these functions satisfy similar properties to those given in \cite{BZ} for $\NH_\alpha(\bx;t)$; taking $t_i =  t^i$ for all $i\in[n-1]$ one recovers the corresponding results in~\cite{BZ}.

It is easy to see $\NH_\alpha(0,\ldots,0) = \bs_\alpha$ and $\NH_\alpha(1,\ldots,1) = \bh_\alpha$. Let $\NSym_n$ be the $n$-th homogeneous component of $\NSym$, which has bases $\{\bs_\alpha:\alpha\models n\}$ and $\{\bh_\alpha:\alpha\models n\}$. Then $\{\NH_\alpha:\alpha\models n\}$ gives a basis for $\NSym_n[t_1,\ldots,t_{n-1}]$, since $\NH_\alpha$ has leading term $\bs_\alpha$ under the partial order $\cleq$ for compositions of $n$. It follows that $\bigsqcup_{n\geq0}\{\NH_\alpha: \alpha\models n\}$ is a basis for $\NSym[t_1,t_2,\ldots]$.

Bergeron and Zabrocki~\cite{BZ} defined an inner product on $\NSym$ such that the basis $\{\bs_\alpha\}$ is ``semi-self'' dual, namely $\langle \bs_\alpha,\bs_\beta \rangle := (-1)^{|\alpha|+\ell(\alpha)} \delta_{\alpha,\,\beta^c}$ where $\delta$ is the Kronecker delta. They showed that the same result holds for $\{\bh_\alpha\}$ and $\{\NH_\alpha(\bx;t)\}$. Now one has a multivariate version.

\begin{proposition}
One has $\langle \NH_\alpha, \NH_\beta \rangle = (-1)^{|\alpha|+\ell(\alpha)} \delta_{\alpha,\,\beta^c}$ for any pair of compositions $\alpha$ and $\beta$.
\end{proposition}

\begin{proof}
By definition, one has
\[
\langle \NH_\alpha, \NH_\beta \rangle = \sum_{\alpha'\cleq\alpha} \underline t^{D(\alpha)\setminus D(\alpha')}  \sum_{\beta'\cleq\beta}  \underline t^{D(\beta)\setminus D(\beta')} \langle \bs_{\alpha'}, \bs_{\beta'} \rangle.
\]
If $|\alpha|\ne|\beta|$ then $\langle \bs_{\alpha'}, \bs_{\beta'} \rangle = 0$ for all $\alpha'\cleq \alpha$ and $\beta'\cleq \beta$. Assume $|\alpha|=|\beta|=n$ below. 

If $D(\alpha)\cup D(\beta) \ne [n-1]$ then again one has$\langle \bs_{\alpha'}, \bs_{\beta'} \rangle = 0$ for all $\alpha'\cleq \alpha$ and $\beta'\cleq \beta$. If $\alpha=\beta^c$ then the right hand side contains only one nonzero term $\langle \bs_{\alpha}, \bs_\beta \rangle$. If $D(\alpha)\cap D(\beta)\ne\emptyset$ then taking $E=D(\alpha)\setminus D(\alpha')$ we write the right hand side as
\[
\sum_{ E\subseteq D(\alpha)\cap D(\beta)} (-1)^{n+\ell(\alpha)-|E|} \,\underline t^{D(\alpha)\cap D(\beta)} =0.
\]
This completes the proof.
\end{proof}

We also give a product formula for $\{\NH_\alpha\}$, generalizing the product formula for $\{ \NH_\alpha(\bx;t)\}$ given by Bergeron and Zabrocki~\cite{BZ}.

\begin{proposition} 
For any compositions $\alpha$ and $\beta$, one has the product formula
\[
\NH_\alpha \cdot \NH_\beta = \sum_{\gamma\cleq\beta} \left(\prod_{i\in D(\beta)\setminus D(\gamma)} (t_i-t_{|\alpha|+i}) \right) \left( \NH_{\alpha\gamma} + (1-t_{|\alpha|}) \NH_{\alpha\rhd\gamma} \right).
\]
\end{proposition}

\begin{proof}
Let $\gamma\cleq\beta$. If $\delta\cleq\alpha\gamma$ then there exists a unique pair of compositions $\alpha'\cleq\alpha$ and $\gamma'\cleq\gamma$ such that $\delta=\alpha'\gamma'$ or $\delta=\alpha'\rhd\gamma'$. If $\delta\cleq\alpha\rhd\gamma$ then there exists a unique pair of compositions $\alpha'\cleq\alpha$ and $\gamma'\cleq\gamma$ such that $\delta=\alpha'\rhd\gamma'$. Thus
\[
\NH_{\alpha\gamma} + (1-t_{|\alpha|}) \NH_{\alpha\rhd\gamma} 
= \sum_{ \substack{ \alpha'\cleq\alpha \\ \gamma'\cleq\gamma }} 
\underline t^{D(\alpha\gamma)\setminus D(\alpha'\gamma')} \bs_{\alpha'\gamma'}
+\underline t^{D(\alpha\gamma)\setminus D(\alpha'\rhd\gamma')} \bs_{\alpha'\rhd\gamma'} 
+(1-t_{|\alpha|}) \,\underline t^{D(\alpha\rhd\gamma)\setminus D(\alpha'\rhd\gamma')} \bs_{\alpha'\rhd\gamma'}.
\]
Since $D(\alpha\gamma) = D(\alpha\rhd\gamma)\sqcup \{|\alpha|\}$ and $D(\alpha\gamma)\setminus D(\alpha'\gamma') = D(\alpha\rhd\gamma)\setminus D(\alpha'\rhd\gamma')$, it follows that
\[
\NH_{\alpha\gamma} + (1-t_{|\alpha|}) \NH_{\alpha\rhd\gamma}  =  
\sum_{ \substack{\alpha'\cleq\alpha \\ \gamma'\cleq\gamma }} \underline t^{D(\alpha\gamma)\setminus D(\alpha'\gamma')} (\bs_{\alpha'\gamma'}+\bs_{\alpha'\rhd\gamma'}).
\]
Note that $\bs_{\alpha'\gamma'}+\bs_{\alpha'\rhd\gamma'} = \bs_{\alpha'}\bs_{\gamma'}$, and $D(\alpha\gamma)\setminus D(\alpha'\gamma') = (D(\alpha)\setminus D(\alpha') ) \sqcup \{\,|\alpha|+i:i\in D(\gamma)\setminus D(\gamma')\}$. Thus the right hand side of the product formula equals
\[
\sum_{ \substack{\alpha'\cleq\alpha \\\gamma'\cleq\beta }} \underline t^{ D(\alpha)\setminus D(\alpha')}  \bs_{\alpha'}\bs_{\gamma'} \sum_{\gamma'\cleq\gamma\cleq\beta} \left(\prod_{i\in D(\beta)\setminus D(\gamma)} (t_i-t_{|\alpha|+i}) \right)  \,\underline t^{|\alpha|+D(\gamma)\setminus D(\gamma')}
\]
where $\underline t^{|\alpha|+S}:=\prod_{i\in S} t_{|\alpha|+i}$. Since the interval $[\gamma',\beta]$ is isomorphic to the Boolean algebra of the subsets of $D(\beta)\setminus D(\gamma')$, one sees that
\[
\sum_{\gamma'\cleq\gamma\cleq\beta} \left(\prod_{i\in D(\beta)\setminus D(\gamma)} (t_i-t_{|\alpha|+i}) \right) \,\underline t^{ |\alpha|+D(\gamma)\setminus D(\gamma')} = \underline t^{D(\beta)\setminus D(\gamma')}.
\]
Therefore the right-hand side of the product formula is equal to
\[
\sum_{\alpha'\cleq\alpha} \underline t^{ D(\alpha)\setminus D(\alpha')} \bs_{\alpha'} \sum_{\gamma'\cleq\beta} \underline t^{D(\beta)\setminus D(\gamma')} \bs_{\gamma'} = \NH_\alpha\cdot \NH_\beta.
\]
The proof is complete.
\end{proof}

In particular, the above proposition implies that $\NH_\alpha(\bx;t_1,\ldots,t_{n-1}) \NH_\beta(\bx;t|n) = \NH_{\alpha \beta} (t|n)$, where $n=|\alpha|$ and $t|n:=\{t_1,\ldots,t_{n-1},1,t_1,\ldots,t_{n-1},1,\ldots\}$, i.e. $t_1,t_2,\ldots$ are $n$-periodic. This allows us to recover the following result of Bergeron and Zabrocki~\cite{BZ}: if $\zeta$ is an $n$-th root of unity then $\NH_\alpha(\bx;\zeta)\NH_\beta(\bx;\zeta)=\NH_{\alpha\beta}(\bx;\zeta)$.

%Coproduct?

\subsection{Quasisymmetric characteristic}
In this subsection we use the two encodings given in \S\ref{sec:multichains} for the multichains in $\cB_n$ to study the quasisymmetric characteristic of the Stanley-Reisner ring $\FF[\cB_n]$.

\begin{lemma}\label{lem:Na}
Let $\alpha$ be a weak composition of $n$. Then the $\alpha$-homogeneous component $\FF[\cB_n]_\alpha$ of the Stanley-Reisner ring $\FF[\cB_n]$ is an $H_n(0)$-submodule of $\FF[\cB_n]$ with homogeneous multigrading $\underline t^{D(\alpha)}$ and isomorphic to the cyclic module $H_n(0)\pi_{w_0(\alpha^c)}$, where $\alpha^c$ is the composition of $n$ with descent set $[n-1]\setminus D(\alpha)$.
\end{lemma}

\begin{proof}
It is not hard to check that $H_n(0)\pi_{w_0(\alpha^c)}$ has an $\FF$-basis 
$
\{\pib_w\pi_{w_0(\alpha^c)}: w\in\SS^\alpha\}.
$
For any $w\in\SS^\alpha$, the $H_n(0)$-action is given by
\[
\pib_i\pib_w\pi_{w_0(\alpha^c)} = 
\begin{cases}
-\pib_w\pi_{w_0(\alpha^c)} & {\rm if}\ i\in D(w^{-1}),\\
0, & {\rm if}\ i\notin D(w^{-1}),\ s_iw\notin\SS^\alpha, \\
\pib_{s_iw}\pi_{w_0(\alpha^c)}, & {\rm if}\ i\notin D(w^{-1}),\ s_iw\in\SS^\alpha.
\end{cases}
\]

On the other hand, if $\MM$ is a multichain of $\cB_n$ with $\alpha(\MM)=\alpha$, then $r(\MM)=D(\alpha)$ and $\sigma(\MM)\in \SS^\alpha$. It follows from (\ref{eq:equiv}) that $\FF[\cB_n]_\alpha\cong H_n(0)\pi_{w_0(\alpha^c)}$ via $y_\MM\mapsto \pib_{\sigma(\MM)}\pi_{w_0(\alpha^c)}$.
\end{proof}

Thus we get an $\NN\times\NN^{n+1}$-multigraded quasisymmetric characteristic for each homogeneous component $\FF[\cB_n]_\alpha$:
\begin{equation}\label{eq:Na}
\mathrm{Ch}_{q,\underline t} (\FF[\cB_n]_\alpha) = \sum_{w\in\SS^\alpha} q^{\inv(w)} \underline t^{D(\alpha)} F_{D(w^{-1})}.
\end{equation}
This defines an $\NN\times\NN^{n+1}$-multigraded quasisymmetric characteristic for the Stanley-Reisner ring $\FF[\cB_n]$.

\begin{theorem}\label{thm:Ch}
The $\NN\times\NN^{n+1}$-multigraded quasisymmetric characteristic $\mathrm{Ch}_{q,\underline t}(\FF[\cB_n])$ of $\FF[\cB_n]$ equals
%\begin{eqnarray}
\[
\sum_{k\geq0} \sum_{\alpha\in\Com(n,k+1)} \underline t^{ D(\alpha)} \sum_{w\in\SS^\alpha} q^{\inv(w)} F_{D(w^{-1})} %\label{eq:Ch1} \\ 
= \sum_{w\in\SS_n} \frac{ q^{\inv(w)} \underline t^{D(w)} F_{D(w^{-1})} }
{\prod_{0\leq i\leq n}(1-t_i)} %\label{eq:Ch2} \\  
= \sum_{k\geq0} \sum_{\pp\in[k+1]^n} t_{p'_1} \cdots t_{p'_k} q^{\inv(\pp)} F_{D(\pp)}. %\label{eq:Ch3}
%\end{eqnarray}
\]
%where $D(w^{-1})$ is identified with the composition $\alpha$ of $n$ for which $D(\alpha)=D(w^{-1})$.
\end{theorem}

\begin{proof}
The first expression of $\mathrm{Ch}_{q,\underline t}(\FF[\cB_n])$ follows immediately from (\ref{eq:Na}). 

To see the second expression, recall that $\FF[\cB_n]$ is a free $\FF[\Theta]$-module on the descent basis $\{Y_w:w\in\SS_n\}$, and the $H_n(0)$-action on $\FF[\cB_n]$ is $\FF [\Theta]$-linear. If $a_0,\ldots,a_n$ are nonnegative integers and $\MM$ is a multichain in $\cB_n$, then one sees that $\theta_0^{a_0}\cdots\theta_n^{a_n} y_\MM $ is the sum of $y_{\MM'}$ for all multichains $\MM'$ refining $\MM$ and having rank multiset $r(\MM')=r(\MM)\cup\{0^{a_0},\ldots,n^{a_n}\}$. Thus for any $w\in\SS_n$, the element $\theta_0^{a_0}\cdots\theta_n^{a_n} Y_w$ has leading term
\[
\prod_{i\in D(w)\cup\{0^{a_0},\ldots,n^{a_n}\}} y_{ \{w(1),\ldots,w(i)\} }.
\] 
It follows that $\theta_0^{a_0}\cdots\theta_n^{a_n} Y_w$ has length-grading $q^{\inv(w)}$. Then one has
\begin{eqnarray*}
\mathrm{Ch}_{q,\underline t} \left(\FF[\cB_n] \right)
&=& {\rm Hilb} \left(\FF [\Theta];\underline t \right) \sum_{w\in\SS_n} q^{\inv(w)} \underline t^{ D(w)} F_{D(w^{-1})}  \\
&=& \sum_{w\in\SS_n} \frac{ q^{\inv(w)} \underline t^{D(w)} F_{D(w^{-1})}  }
{\prod_{0\leq i\leq n}(1-t_i)}.
\end{eqnarray*}

Finally we encode a multichain $\MM$ of length $k$ in $\cB_n$ by $p(\MM)=\pp\in[k+1]^n$. The $H_n(0)$-action in terms of this encoding is equivalent to the first one via (\ref{eq:equiv}). One has $D(\alpha(\MM))$ equals the multiset of $p'_1,\ldots,p'_k$ and $\inv(\sigma(\MM))=\inv(p(\MM))$. Hence we get the third expression of $\mathrm{Ch}_{q,\underline{t}}(\FF[\cB_n])$.
\end{proof}

\subsection{Applications to permutation statistics}\label{sec:Ch}

Theorem~\ref{thm:Ch} specializes to the result of Garsia and Gessel~\cite[Theorem~2.2]{GarsiaGessel} on the multivariate generating function of the permutation statistics $\inv(w)$, $\maj(w)$, $\des(w)$, $\maj(w^{-1})$, and $\des(w^{-1})$ for all $w\in \SS_n$. To explain this, we first recall that
\[
F_\alpha = \sum_{\substack{i_1\geq\cdots\geq i_n\geq1\\ i\in D(\alpha)\Rightarrow i_j> i_{j+1}}} x_{i_1}\cdots x_{i_n},\quad \forall \alpha\models n.
\]
Given a nonnegative integer $\ell$, let $\ps_{q;\ell}$ be the linear transformation from formal power series in $x_1,x_2,\ldots$ to formal power series in $q$, defined by $\ps_{q;\ell}(x_i)=q^{i-1}$ for $i=1,\ldots,\ell$, and $\ps_{q;\ell}(x_i)=0$ for all $i>\ell$; similarly, $\ps_{q;\infty}$ is defined by $\ps_{q;\infty}(x_i)=q^{i-1}$ for all $i=1,2,\ldots$. It is well known (see Stanley~\cite[Lemma~7.19.10]{EC2}) that 
\[
\ps_{q;\infty} (F_\alpha) = \frac{ q^{\maj(\alpha)}} {(1-q)\cdots(1-q^n)}.
\]
Let $(u;q)_n:=(1-u)(1-qu)(1-q^2u)\cdots(1-q^nu)$. It is also not hard to check (see Gessel and Reutenauer~\cite[Lemma~5.2]{GesselReutenauer}) that 
\[
\sum_{\ell\geq0}u^\ell\ps_{q;\ell+1}(F_\alpha) = \frac{q^{\maj(\alpha)} u^{\des(\alpha)} } {(u;q)_n}.
\]

A \emph{bipartite partition} is a pair of weak compositions $\lambda=(\lambda_1,\ldots,\lambda_n)$ and $\mu=(\mu_1,\ldots,\mu_n)$ satisfying the conditions $\lambda_1\geq\cdots\geq\lambda_n$ and $\lambda_i=\lambda_{i+1} \Rightarrow \mu_i\geq \mu_{i+1}$ (so the pairs of nonnegative integers $(\lambda_1,\mu_1),\ldots,(\lambda_n,\mu_n)$ are lexicographically ordered). Let $B(\ell,k)$ be the set of bipartite partitions $(\lambda,\mu)$ such that $\max(\lambda)\leq\ell$ and $\max(\mu)\leq k$, where $\max(\mu):=\max\{\mu_1,\ldots,\mu_n\}$ and similarly for $\max(\lambda)$.

\begin{corollary}[Garsia and Gessel~\cite{GarsiaGessel}]
\[
\frac{ \sum_{w\in\SS_n} q_0^{\inv(w)} q_1^{\maj(w^{-1})} u_1^{\des(w^{-1})} q_2^{\maj(w)} u_2^{\des(w)} } {(u_1;q_1)_n (u_2;q_2)_n }
= \sum_{\ell,k\geq0} u_1^\ell u_2^k \sum_{ (\lambda,\mu)\in B(\ell,k) } q_0^{\inv(\mu)} q_1^{|\lambda|} q_2^{|\mu|}.
\]
\end{corollary}

\begin{proof}
Applying the linear transformation $\sum_{\ell\geq0}u_1^{\ell}\ps_{q_1;\ell+1}$ and the specialization $t_i=q_2^i u_2$ for $i=0,1,\ldots,n$ to the last two expressions of $\mathrm{Ch}_{q,\underline{t}}(\FF[\cB_n])$ in Theorem~\ref{thm:Ch}, we obtain
\[
\frac{ \sum_{w\in\SS_n} q_0^{\inv(w)} q_1^{\maj(w^{-1})} u_1^{\des(w^{-1})} q_2^{\maj(w)} u_2^{\des(w)} } {(u_1;q_1)_n (u_2;q_2)_n }
= \sum_{k\geq0} u_2^k \sum_{\pp\in[k+1]^n} q_2^{p'_1+\cdots p'_k} q_0^{\inv(\pp)} \sum_{\ell\geq0} u_1^\ell \sum_{\substack{\ell\geq \lambda_1\geq \cdots \geq \lambda_n\geq0 \\ j\in D(\pp)\Rightarrow \lambda_j>\lambda_{j+1}}} q_1^{\lambda_1+\cdots+\lambda_n} .
\]
Note that $\pp\in[k+1]^n$ if and only if $\mu:=(k+1-p_1,\ldots,k+1-p_n)$ is a weak composition with $\max(\mu)\leq k$, and one has $|\mu| = p'_1+\cdots+p'_k$ by the definition of $p'_i$. The condition $j\in D(\pp)\Rightarrow \lambda_j>\lambda_{j+1}$ is equivalent to $\lambda_i=\lambda_{i+1} \Rightarrow \mu_i\geq \mu_{i+1}$. Thus we can rewrite the right hand side as a sum over $(\lambda,\mu)\in B(\ell,k)$ for all $\ell,k\geq0$, and then the result follows easily.
\end{proof}

Taking $q_0=1$ in Theorem~\ref{thm:Ch} one has the usual $\NN^{n+1}$-multigraded quasisymmetric characteristic $\mathrm{Ch}_{\underline{t}}(\FF[\cB_n])$. Then applying the same specialization as in the proof of the above corollary, and using the observation 
\[
\sum_{(\lambda,\mu)\in B(\ell,k)} q_1^{|\lambda|} q_2^{|\mu|} = \prod_{0\leq i\leq \ell} \prod_{0\leq j\leq k} \frac{1}{1-zq_1^iq_2^j} \,\Bigg|_{z^n}
\] 
where $f|_{z^n}$ is the coefficient of $z^n$ in $f$, one can get another result of Garsia and Gessel~\cite{GarsiaGessel}:
\[
\frac{ \sum_{w\in\SS_n} q_1^{\maj(w^{-1})} u_1^{\des(w^{-1})} q_2^{\maj(w)} u_2^{\des(w)} } {(u_1;q_1)_n (u_2;q_2)_n  }
= \sum_{\ell,k\geq0} u_1^\ell u_2^k  \prod_{0\leq i\leq \ell} \prod_{0\leq j\leq k} \frac{1}{1-zq_1^iq_2^j} \,\Bigg|_{z^n}.
\]

%\frac{ \sum_{n\geq0} u^n \sum_{w\in\SS_n} x^{\maj(w^{-1})} y^{\des(w^{-1})} t^{\maj(w)} s^{\des(w)} } {(1-y)(1-xy)\cdots(1-x^ny) (1-s)(1-ts)\cdots(1-t^ns) } = \sum_{\ell,k\geq0} y^\ell s^k  \prod_{0\leq i\leq \ell} \prod_{0\leq j\leq k} \frac{1}{1-ux^it^j}.

A further specialization of Theorem~\ref{thm:Ch} gives a well known result which is often attributed to Carlitz~\cite{Carlitz} but actually dates back to MacMahon~\cite[Volume 2, Chapter 4]{MacMahon}. 

\begin{corollary}[MacMahon-Carlitz]
Let $[k+1]_q:=1+q+q^2+\cdots+q^k$. Then
\[
\frac{\sum_{w\in\SS_n} q^{\maj(w)} u^{\des(w)} } {(u;q)_n} = \sum_{k\geq0}([k+1]_q)^n u^k.
\]
\end{corollary}

\begin{proof}
Taking $q=1$, $t_i=q^iu$ for all $i=0,1,\ldots,n$, and $F_I=1$ for all $I\subseteq[n-1]$ in last two expressions of $\mathrm{Ch}_{q,\underline{t}}(\FF[\cB_n])$ in Theorem~\ref{thm:Ch}, we get
\[
\frac{\sum_{w\in\SS_n} q^{\maj(w)} u^{\des(w)} } {(u;q)_n} = \sum_{k\geq0} u^k \sum_{\pp\in[k+1]^n} q^{p'_1+\cdots+p'_k}.
\]
Then using Equation (\ref{eq:box}) we establish this corollary.
\end{proof}

Now we switch to the first two expressions of $\mathrm{Ch}_{q,\underline{t}}(\FF[\cB_n])$ and derive the following result from it, which was obtained by Adin, Brenti, and Roichman~\cite{AdinBrentiRoichman} from the Hilbert series of the coinvariant algebra $\FF[\px]/(\FF[\px]^{\SS_n}_+)$.

\begin{corollary}[Adin, Brenti, and Roichman~\cite{AdinBrentiRoichman}]
Denote by $\Par(n)$ the set of all weak partitions 
$\lambda=(\lambda_1,\ldots,\lambda_n)$ with
$\lambda_1\geq\cdots\geq\lambda_n\geq0$, and let 
$m(\lambda)=(m_0(\lambda),m_1(\lambda),\ldots)$, where
\[
m_j(\lambda):=\#\{1\leq i\leq n:\lambda_i=j\}.
\]
Then 
\[
\sum_{\lambda\in\Par(n)}{n\choose m(\lambda)}\prod_{i=1}^n q_i^{\lambda_i} =\frac{\sum_{w\in\SS_n}\prod_{i\in D(w)}q_1\cdots q_i}
{(1-q_1)(1-q_1q_2)\cdots(1-q_1\cdots q_n)}.
\]
\end{corollary}

\begin{proof}
Given an integer $k\geq0$, the multichains $\MM=(A_1\subseteq\cdots\subseteq A_k)$ with $A_1\ne\emptyset$ are in bijection with the pairs $(\alpha(\MM),\sigma(\MM))$ of $\alpha(\MM)\in\Com_1(n,k+1)$ and $\sigma(\MM)\in\SS^\alpha$, where
\[
\Com_1(n,k+1):=\{(\alpha_1,\ldots,\alpha_{k+1})\in\Com(n,k+1):\alpha_1\geq1\}.
\]
Hence the proof of Theorem~\ref{thm:Ch} implies
\[
\sum_{k\geq0} \sum_{\alpha\in\Com_1(n,k+1)} \underline t^{D(\alpha)} \sum_{w\in\SS^\alpha} F_{D(w^{-1})} =\frac{\sum_{w\in\SS_n} \underline t^{D(w)} F_{D(w^{-1})}  }{\prod_{1\leq i\leq n} (1-t_i)}.
\]
Taking $t_i=q_1\cdots q_i$ for $i=1,\ldots,n$, and $F_{D(w^{-1})}=1$ for all $w\in\SS_n$, we obtain
\[
\sum_{k\geq0} \sum_{\alpha\in\Com_1(n,k+1)} {n\choose\alpha} \prod_{i\in D(\alpha)}q_1\cdots q_i =\frac{\sum_{w\in\SS_n}\prod_{i\in D(w)}q_1\cdots q_i}{(1-q_1)\cdots(1-q_1\cdots q_n)}.
\]
Thus it remains to show
\[
\sum_{\lambda\in\Par(k,n)} {n\choose m(\lambda)}\prod_{i=1}^n q_i^{\lambda_i} = 
\sum_{\alpha\in\Com_1(n,k+1)} {n\choose\alpha} \prod_{j\in D(\alpha)}q_1\cdots\ q_j
\]
where $\Par(k,n):=\{(\lambda_1,\ldots,\lambda_n)\in\Par(n):\lambda_1=k\}$, for all $k\geq0$. This can be established by using the bijection $\lambda\mapsto\alpha(\lambda):=(m_k(\lambda),\ldots,m_0(\lambda))$ between $\Par(k,n)$ and $\Com_1(n,k+1)$. One sees that the multiset $D(\alpha(\lambda))$ is precisely the multiset of column lengths of the Young diagram of $\lambda$, and thus
\[
\lambda_i=\#\{j\in D(\alpha(\lambda)):j\geq i\},\quad\forall i\in[n].
\] 
This completes the proof.
\end{proof}

\section{Remarks and questions for future research}\label{sec:Remarks}

\subsection{Connection with the polynomial ring}\label{sec:transfer}
We give in \S\ref{sec:rank} an analogy via the transfer map $\tau$ between the rank-selected Stanley-Reisner ring $\FF[\cB_n^*]$ as a multigraded algebra and $\SS_n$-module and the polynomial ring $\FF[\px]$ as a graded algebra and $\SS_n$-module. With our $H_n(0)$-action on $\FF[\cB_n^*]$ and the usual $H_n(0)$-action on $\FF[\px]$, the transfer map $\tau$ is not an isomorphism of $H_n(0)$-modules: e.g. for $n=2$ one has 
\[
\pib_1(y_1^2) = y_2^2,\quad \pib_1(x_1^2)=x_2^2+x_1x_2\ne x_2^2 = \tau(y_2^2).
\]

However, there is still a similar analogy between the multigraded $H_n(0)$-module $\FF[\cB_n^*]$ and the graded $H_n(0)$-module $\FF[\px]$. In fact, Theorem~\ref{thm:NHL} and Theorem~\ref{thm:Ch} imply
\[
\mathbf{ch}_{\underline t}(\FF[\cB_n^*]) = \frac{ \sum_{\alpha\models n} \underline t^{D(\alpha)} \bs_\alpha } { \prod_{1\leq i\leq n} (1-t_i) },
\]
\[
\mathrm{Ch}_{q,\underline t}(\FF[\cB_n^*]) = \frac{  \sum_{w\in\SS_n} q^{\inv(w)} \underline t^{D(w)} F_{D(w^{-1})} } { \prod_{1\leq i\leq n} (1-t_i) }.
\]
They specialize to the graded noncommutative characteristic and bigraded quasisymmetric characteristic of $\FF[\px]$ via $t_i=t^i$ for $i=1,\ldots,n$, as it follows from our earlier work \cite{H} that
\[
\mathbf{ch}_t(\FF[\px]) = \frac{ \sum_{\alpha\models n} t^{\maj(\alpha)} \bs_\alpha } { \prod_{1\leq i\leq n} (1-t^i) },
\]
\[
\mathrm{Ch}_{q,t}(\FF[\px]) =  \frac{ \sum_{w\in\SS_n} q^{\inv(w)} t^{\maj(w)} F_{D(w^{-1})} } { \prod_{1\leq i\leq n} (1-t^i) }.
\]

This suggests an isomorphism $\FF[\cB_n^*]\cong \FF[\px]$ of graded $H_n(0)$-modules. To explicitly give such an isomorphism, we consider every $\alpha$-homogeneous component of $\FF[\cB_n^*]$, which has a basis 
$
\left\{ \pib_w(\prod_{i\in D(\alpha)} y_{\{1,\ldots,i\}}): w\in\SS^\alpha \right\}.
$
By sending this basis to 
$
\left\{ \pib_w (\prod_{i\in D(\alpha)} x_1\cdots x_i): w\in\SS^\alpha \right\}
$
which, by our work~\cite{H}, is triangularly related to 
$
\left\{ \prod_{i\in D(\alpha)} x_{w(1)}\cdots x_{w(i)}: w\in\SS^\alpha \right\},
$
one has the desired isomorphism.

\subsection{Hecke algebra action on the Stanley-Reisner ring of the Coxeter complex}\label{sec:HWSR}
Let 
\[
W:=\langle S:s_i^2=1,\ (s_is_js_i\cdots)_{m_{ij}} = (s_js_is_j\cdots)_{m_{ij}},1\leq i\ne j\leq d \rangle
\]
be a finite Coxeter group generated by $S=\{s_1,\ldots,s_d\}$, where $m_{ij}\in\{2,3,\ldots\}$ and $(aba\cdots)_m$ is an alternating product of $m$ terms. The \emph{Hecke algebra} $H_W(q)$ of $W$ is the $\FF(q)$-algebra generated by $T_1,\ldots,T_d$ with relations
\[
\left\{\begin{array}{ll}
(T_i+1)(T_i-q)=0, & 1\leq i\leq d,\\
(T_iT_jT_i\cdots)_{m_{ij}}=(T_jT_iT_j\cdots)_{m_{ij}}, & 1\leq i\ne
j\leq d.
\end{array}\right.
\]
The elements $T_w:=T_{i_1}\cdots T_{i_k}$ are well defined for all $w\in W$ with a reduced expression $w=s_{i_1}\cdots s_{i_k}$, and form an $\FF(q)$-basis for $H_W(q)$. Specializing $q=1$ one has the group algebra $\FF W$, with $s_i=T_i|_{q=0}$ for all $i\in [d]$. Specializing $q=0$ one has the \emph{$0$-Hecke algebra} $H_W(0)$ generated by $\pib_i:=T_i|_{q=0}$ for all $i\in[d]$. The elements $\pi_1,\ldots,\pi_d$ form another generating set for $H_W(0)$, where $\pi_i:=\pib_i+1$. The elements $\pib_w$ and $\pi_w$ are well defined for all $w\in W$, giving two bases $\{\pib_w:w\in W\}$ and $\{\pi_w:w\in W\}$ for $H_W(q)$. By Norton~\cite{Norton}, the $0$-Hecke algebra $H_W(0)$ has the same representation theory as $H_n(0)$; one only needs to replace compositions with subsets of $S$.

The symmetric group $\SS_n$ is the Coxeter group of type $A_{n-1}$. The Stanley-Reisner ring of $\cB_n$ is essentially the Stanley-Reisner ring of the Coxeter complex of $\SS_n$. We can generalize our action $H_n(0)$-action on $\FF[\cB_n]$ to an $H_W(q)$-action on the Stanley-Reisner ring $\FF(q)[\Delta(W)]$ of the Coxeter complex $\Delta(W)$ of $W$. One has similar results for this $H_W(q)$-action, from which one can recover most results in Section~\ref{sec:H0SR} by taking $W=\SS_n$ and $q=0$. It is somewhat technical to provide all the details, but we can at least give a sketch here.

The Coxeter complex $\Delta(W)$ of a finite Coxeter group $W$ is a simplicial complex whose faces are the parabolic cosets $wW_J$ for all $w\in W$ and $J\subseteq S$, ordered by reserve inclusion. The vertices of $\Delta(W)$ are the maximal parabolic cosets $wW_{i^c}$ for all $w\in W$ and $i\in[d]$, where $i^c:=S\setminus\{s_i\}$, and they are colored by $r(wW_{i^c})=i$ so that $\Delta(W)$ is balanced. A face $wW_J$ has vertices $wW_{i^c}$ for all $i\in J^c:=S\setminus J$, and thus has \emph{rank set} $r(wW_J)=J^c$. This defines a multigrading on the Stanley-Reisner ring $\FF[\Delta(W)]$. 

The $W$-action on its parabolic cosets induces a $W$-action on the Stanley-Reisner ring $\FF[\Delta(W)]$, preserving its multigrading. One can show that the invariant algebra of this $W$-action equals the polynomial algebra $\FF[\Theta]$ (c.f. Garsia and Stanton~\cite{GarsiaStanton}), and the $W$-action is $\Theta$-linear. Here $\Theta$ is the set of the rank polynomials 
\[
\theta_i=\sum_{w\in W^{i^c}} wW_{i^c},\quad i=1,\ldots,d.
\]

Let $J\subseteq S$. The rank-selected subcomplex $\Delta_J(W)$ consists of all faces whose rank set is contained in $J$. The Stanley-Reisner ring $\FF[\Delta_J(W)]$ inherits a multigrading and a $W$-action from $\FF[\Delta(W)]$. Let $\Theta_J:=\{\theta_j:s_j\in J\}$. Then a theorem of Kind and Kleinschmidt~\cite{KindKleinschmidt} implies that $\FF[\Delta_J(W)]$ is a free $\FF[\Theta_J]$-module with a basis of the descent monomials 
\[
wW_{D(w)^c} = \prod_{i\in D(w)} wW_{i^c}, \quad \forall w\in W^{J^c}.
\] 
The $W$-action on $\FF[\Delta_J(W)]$ is $\Theta_J$-linear, and thus reduces to the quotient algebra $\FF[\Delta_J(W)]/(\Theta_J)$. 

Now we sketch how we define an $H_W(q)$-action on $\FF(q)[\Delta(W)]$. For every $J\subseteq S$, the element $\sigma_J:=\sum_{w\in W_J} T_w$ generates the \emph{parabolic representation} $H_W(q)\sigma_J$ of $H_W(q)$, which has a natural $\FF$-basis $\{T_w \sigma_J: w\in W^J\}$ (see e.g. Mathas~\cite{Mathas}). If $m=y_{v_1}\cdots y_{v_k}$ is a nonzero monomial in $\FF(q)[\Delta(W)]$, then the set of the vertices $v_1,\ldots,v_k$ is a face of $\Delta(W)$, i.e. a coset $wW_J$ where $J\subseteq S$ and $w\in W^J$. We then let $H_W(q)$ act on $m$ in the same way as it acts on $T_w\sigma_J$, i.e. we define
\[
T_i(m):=
\begin{cases}
(q-1)m+qs_i(m), & {\rm if}\ i\in D(w^{-1}),\\
qm, & {\rm if}\ i\notin D(w^{-1}),\ s_iw\notin W^J,\\
s_i(m), & {\rm if}\ i\notin D(w^{-1}),\ s_iw\in W^J.
\end{cases}
\]
We show that this gives an $H_W(q)$-action on $\FF(q)[\Delta(W)]$, with the \emph{invariant algebra} (the trivial isotypic component on which every $T_i$ acts by $q$) 
\[
\FF(q)[\Delta(W)]^{H_W(q)}:=\{ f\in \FF(q)[\Delta(W)]: T_i f =qf,\ 1\leq i\leq d\}
\]
equal to the polynomial algebra $\FF(q)[\Theta]$. We also show that our $H_W(q)$-action on $\FF(q)[\Delta(W)]$ is $\Theta$-linear and preserves the multigrading, hence reducing to $\FF(q)[\Delta_J(W)]/(\Theta_J)$ for all $J\subseteq S$. Furthermore, we obtain the following results.

\vskip5pt\noindent(i) 
\emph{There is an $H_W(q)$-module isomorphism $\FF(q)[\Delta(W)]/(\Theta)\cong H_W(q)$ if $q$ is generic, i.e. if $q$ is an indeterminate or $q\in \FF\setminus E$ for some finite set $E\subsetneq\FF$ ($E$ depends on $\FF$ and is not explicitly known).}

\vskip5pt\noindent(ii) 
\emph{Let $q=0$ and $J\subseteq S$. Then there is an $H_W(0)$-module decomposition 
\[
\FF[\Delta_J(W)]/(\Theta_J) = \bigoplus_{I\subseteq J} Q_I.
\]
Each summand $Q_I$ is the $\FF$-span of $\{wW_{I^c}: D(w)=I\}$ inside $\FF[\Delta_J(W)]/(\Theta_J)$, has homogeneous multigrading $\underline t^I$, and is isomorphic to the projective indecomposable $H_W(0)$-module $\P_I:=H_W(0)\pib_{w_0(I)}\pi_{w_0(I^c)}$. }

\vskip5pt\noindent(iii)
\emph{In particular, there is an $H_W(0)$-module isomorphism $\FF[\Delta(W)]/(\Theta) \cong H_W(0)$ for any field $\FF$.}

\vskip5pt
Since the Coxeter complex $\Delta(\SS_n)$ is the order complex of $\cB_n\setminus\{\emptyset,[n]\}$, we get an $H_n(q)$-action on $\FF(q)[\cB_n]$ by taking $W=\SS_n$. We recover most results in Section~\ref{sec:H0SR} by further specializing $q=0$.

%\subsection{Lattice points in certain cones}
%Beck and Braun~\cite{BeckBraun} obtained
%\sum_{w\in\SS_n}\frac{\prod_{i\in D(w)}z_0z_{w(1)}\cdots z_{w(i)}} {\prod_{j=0}^n (1-z_0z_{w(1)}\cdots z_{w(j)})} = \sum_{k\geq0}\left(\prod_{j=1}^n[k+1]_{z_j}\right)z_0^k
%by counting the lattice points contained in the cone in $\mathbb R^{n+1}$ spanned by $\{(1,z_1,\ldots,z_n):0\leq z_i\leq 1\}$. As remarked in \cite{BeckBraun}, this geometric approach is related to the theory of $P$-partitions, and alternatively, one can get the same generating function from the multigraded Hilbert series of the affine semigroup algebra of the semigroup of lattice points in the aforementioned cone. We observe that there is a natural $H_n(0)$-action on this semigroup algebra, preserving its multigrading and giving a multigraded quasisymmetric characteristic which specializes to the above generating function.

\subsection{Gluing the group algebra and the $0$-Hecke algebra}
The group algebra $\FF W$ of a finite Coxeter group $W$ naturally admits both actions of $W$ and $H_W(0)$. Hivert and Thi\'ery~\cite{HivertThiery} defined the \emph{Hecke group algebra} of $W$ by gluing these two actions. In type $A$, one can also glue the usual actions of $\SS_n$ and $H_n(0)$ on the polynomial ring $\FF[\px]$, but the resulting algebra is different from the Hecke group algebra of $\SS_n$. 

Now one has a $W$-action and an $H_W(0)$-action on the Stanley-Reisner ring $\FF[\Delta(W)]$. What can we say about the algebra generated by the operators $s_i$ and $\overline\pi_i$ on $\FF[\Delta(W)]$? Is it the same as the Hecke group algebra of $W$? If not, what properties (dimension, bases, presentation, simple and projective indecomposable modules, etc.) does it have?

\subsection{Tits Building}
Let $\Delta(G)$ be the Tits building of the general linear group $G=GL(n,\FF_q)$ and its usual BN-pair over a finite field $\FF_q$; see e.g. Bj\"orner~\cite{Bjorner}. The Stanley-Reisner ring $\FF[\Delta(G)]$ is a $q$-analogue of $\FF[\cB_n]$. The nonzero monomials in $\FF[\Delta(G)]$ are indexed by multiflags of subspaces of $\FF_q^n$, and there are $q^{\inv(w)}$ many multiflags corresponding to a given multichain $\MM$ in $\cB_n$, where $w=\sigma(\MM)$. Can one obtain the multivariate quasisymmetric function identities in Theorem~\ref{thm:Ch} by defining a nice $H_n(0)$-action on $\FF[\Delta(G)]$?

\end{document}